\pdfoutput=1

\documentclass[titlepage]{amsart}

\usepackage{amsaddr,amsrefs,amssymb,float,tikz}
\usepackage[a4paper,margin=.925in]{geometry}

\usepackage[linktocpage=true, colorlinks=true, linkcolor=blue, citecolor=red, urlcolor=green]{hyperref}

\usetikzlibrary{arrows.meta,decorations}

\tikzset{%
    add/.style args={#1 and #2}{
        to path={%
 ($(\tikztostart)!-#1!(\tikztotarget)$)--($(\tikztotarget)!-#2!(\tikztostart)$)%
  \tikztonodes},add/.default={.2 and .2}}
}

\pgfdeclaredecoration{arrows}{draw}{
\state{draw}[width=\pgfdecoratedinputsegmentlength]{%
  \path [every arrow subpath/.try] \pgfextra{%
    \pgfpathmoveto{\pgfpointdecoratedinputsegmentfirst}%
    \pgfpathlineto{\pgfpointdecoratedinputsegmentlast}%
   };
}}

\tikzset{every arrow subpath/.style={->, draw, thick}}

\newcommand{\foresix}[8]{
\begin{tikzpicture}[baseline=(a.base),scale=#1,every node/.style={scale=#2},inner sep=0pt,outer sep=0pt]
\node (a) at (2,0) {$#8$};
\node at (3,0) {$#7$};
\node at (4,0) {$#6$};
\node at (4,-1.5) {$#4$};
\node at (5,0) {$#5$};
\node at (6,0) {$#3$};
\end{tikzpicture}
}

\newcommand{\foreseven}[9]{
\begin{tikzpicture}[baseline=(a.base),scale=#1,every node/.style={scale=#2},inner sep=0pt,outer sep=0pt]
\node (a) at (1,0) {$#9$};
\node at (2,0) {$#8$};
\node at (3,0) {$#7$};
\node at (4,0) {$#6$};
\node at (4,-1.5) {$#4$};
\node at (5,0) {$#5$};
\node at (6,0) {$#3$};
\end{tikzpicture}
}

\newcommand{\foreight}[9]{
\begin{tikzpicture}[baseline=(a.base),scale=.1,every node/.style={scale=#1},inner sep=0pt,outer sep=0pt]
\node (a) at (1,0) {$#9$};
\node at (2,0) {$#8$};
\node at (3,0) {$#7$};
\node at (4,0) {$#6$};
\node at (5,0) {$#5$};
\node at (5,-1.5) {$#3$};
\node at (6,0) {$#4$};
\node at (7,0) {$#2$};
\end{tikzpicture}
}

\tikzset{DynkinNode/.style={circle,minimum size=.7em,inner sep=0pt,font=\scriptsize}}

\newcounter{bcount}
\newcounter{dcount}
\newcounter{lastk}

\newcommand{\ADynkin}[1]{
\begin{tikzpicture}[scale=.5,baseline=(current bounding box.183)]
\foreach\kthweight[count=\k] in {#1}{
\pgfmathtruncatemacro{\prevnode}{\k-1}
\node[DynkinNode] (\k) at (\k,0) {$\scriptscriptstyle\kthweight$};
\ifnum\k>1\draw[thick] (\k) -- (\prevnode);\fi
}
\end{tikzpicture}
}

\newcommand{\oneADynkin}[1]{
\begin{tikzpicture}[scale=.5,baseline=(current bounding box.mid west)]
\foreach\kthweight[count=\k] in {#1}{
 \ifnum\k=1
  \setcounter{bcount}\kthweight
 \fi
 \ifnum\k>1
  \setcounter{lastk}\k
  \pgfmathtruncatemacro{\prevnode}{\k-1}
  \node[DynkinNode] (\k) at (\k,0) {$\scriptscriptstyle\kthweight$};
  \ifnum\k>2
   \draw[thick] (\k) -- (\prevnode);
  \fi
 \fi
}
\pgfmathparse{1+\thelastk/2}
\node[DynkinNode] (1) at (\pgfmathresult,1) {$\scriptscriptstyle\thebcount$};
\draw[thick] (2) -- (1) -- (\thelastk);
\end{tikzpicture}
}

\newcommand{\oneCDynkin}[2]{
\begin{tikzpicture}[scale=.5,baseline=(current bounding box.183)]
\foreach\kthweight[count=\k] in {#1}{
\setcounter{bcount}\k
\pgfmathtruncatemacro{\prevnode}{\k-1}
\node[DynkinNode] (\k) at (\k,0) {$\scriptscriptstyle\kthweight$};
\ifnum\k=2\draw[-Implies,double,thick] (\prevnode) -- (\k);\fi
\ifnum\k>2\draw[thick] (\k) -- (\prevnode);\fi
}
\setcounter{lastk}\thebcount
\stepcounter{bcount}
\node[DynkinNode] (\thebcount) at (\thebcount,0) {$\scriptscriptstyle#2$};
\draw[-Implies,double,thick] (\thebcount) -- (\thelastk);
\end{tikzpicture}
}

\newcommand{\oneBDynkin}[2]{
\begin{tikzpicture}[scale=.5,baseline=(current bounding box.183)]
\foreach\kthweight[count=\k] in {#1}{
\setcounter{bcount}\k
\ifnum\k=1\node[DynkinNode] (\k) at (2,.78) {$\scriptscriptstyle\kthweight$};\fi
\ifnum\k=2\node[DynkinNode] (\k) at (2,-.78) {$\scriptscriptstyle\kthweight$};\fi
\pgfmathtruncatemacro{\prevnode}{\k-1}
\ifnum\k>2\node[DynkinNode] (\k) at (\k,0) {$\scriptscriptstyle\kthweight$};\fi
\ifnum\k>3\draw[thick] (\k) -- (\prevnode);\fi
}
\setcounter{lastk}\thebcount
\stepcounter{bcount}
\node[DynkinNode] (\thebcount) at (\thebcount,0) {$\scriptscriptstyle#2$};
\draw[-Implies,double,thick] (\thelastk) -- (\thebcount);
\draw[thick] (2) -- (3) -- (1);
\end{tikzpicture}
}

\newcommand{\BDynkin}[2]{
\begin{tikzpicture}[scale=.5,baseline=(current bounding box.183)]
\foreach\kthweight[count=\k] in {#1}{
\setcounter{bcount}\k
\pgfmathtruncatemacro{\prevnode}{\k-1}
\node[DynkinNode] (\k) at (\k,0) {$\scriptscriptstyle\kthweight$};
\ifnum\k>1\draw[thick] (\k) -- (\prevnode);\fi
}
\setcounter{lastk}\thebcount
\stepcounter{bcount}
\node[DynkinNode] (\thebcount) at (\thebcount,0) {$\scriptscriptstyle#2$};
\draw[-Implies,double,thick] (\thelastk) -- (\thebcount);
\end{tikzpicture}
}

\newcommand{\CDynkin}[2]{
\begin{tikzpicture}[scale=.5,baseline=(current bounding box.183)]
\foreach\kthweight[count=\k] in {#1}{
\setcounter{bcount}\k
\pgfmathtruncatemacro{\prevnode}{\k-1}
\node[DynkinNode] (\k) at (\k,0) {$\scriptscriptstyle\kthweight$};
\ifnum\k>1\draw[thick] (\k) -- (\prevnode);\fi
}
\setcounter{lastk}\thebcount
\stepcounter{bcount}
\node[DynkinNode] (\thebcount) at (\thebcount,0) {$\scriptscriptstyle#2$};
\draw[-Implies,double,thick] (\thebcount) -- (\thelastk);
\end{tikzpicture}
}

\newcommand{\oneDDynkin}[3]{
\begin{tikzpicture}[scale=.4,baseline=(current bounding box.183)]
\foreach\kthweight[count=\k] in {#1}{
\setcounter{bcount}\k
\pgfmathtruncatemacro{\prevnode}{\k-1}
\ifnum\k=1\node[DynkinNode] (\k) at (2,.78) {$\scriptscriptstyle\kthweight$};\fi
\ifnum\k=2\node[DynkinNode] (\k) at (2,-.78) {$\scriptscriptstyle\kthweight$};\fi
\ifnum\k>2\node[DynkinNode] (\k) at (\k,0) {$\scriptscriptstyle\kthweight$};\fi
\ifnum\k>3\draw[thick] (\k) -- (\prevnode);\fi
}
\setcounter{lastk}\thebcount
\stepcounter{bcount}
\setcounter{dcount}\thebcount
\stepcounter{bcount}
\node[DynkinNode] (\thedcount) at (\thedcount,.78) {$\scriptscriptstyle#2$};
\node[DynkinNode] (\thebcount) at (\thedcount,-.78) {$\scriptscriptstyle#3$};
\draw[thick] (\thebcount) -- (\thelastk) -- (\thedcount);
\draw[thick] (1) -- (3) -- (2);
\end{tikzpicture}
}

\newcommand{\DDynkin}[3]{
\begin{tikzpicture}[scale=.4,baseline=(current bounding box.185)]
\foreach\kthweight[count=\k] in {#1}{
\setcounter{bcount}\k
\pgfmathtruncatemacro{\prevnode}{\k-1}
\node[DynkinNode] (\k) at (\k,0) {$\scriptscriptstyle\kthweight$};
\ifnum\k>1\draw[thick] (\k) -- (\prevnode);\fi
}
\setcounter{lastk}\thebcount
\stepcounter{bcount}
\setcounter{dcount}\thebcount
\stepcounter{bcount}
\node[DynkinNode] (\thedcount) at (\thedcount,.78) {$\scriptscriptstyle#2$};
\node[DynkinNode] (\thebcount) at (\thedcount,-.78) {$\scriptscriptstyle#3$};
\draw[thick] (\thebcount) -- (\thelastk) -- (\thedcount);
\end{tikzpicture}
}

\newcommand{\oneGDynkin}[1]{
\begin{tikzpicture}[scale=.8,anchor=base,baseline,inner sep=.5pt]
\foreach\kthweight[count=\k] in {#1}{
\ifnum\k=1\node[DynkinNode] (0) at (0,0) {$\scriptscriptstyle\kthweight$};\fi
\ifnum\k=2\node[DynkinNode] (1) at (1,0) {$\scriptscriptstyle\kthweight$};\fi
\ifnum\k=3\node[DynkinNode] (2) at (2,0) {$\scriptscriptstyle\kthweight$};\fi
}
\draw[thick] (0) -- (1);
\draw[-Implies,double distance=1.5pt] (1) -- (2);
\draw (1) -- (2);
\end{tikzpicture}
}

\newcommand{\GDynkin}[1]{
\begin{tikzpicture}[scale=.8,anchor=base,baseline,inner sep=.5pt]
\foreach\kthweight[count=\k] in {#1}{
\ifnum\k=1\node[DynkinNode] (1) at (-1,0) {$\scriptscriptstyle\kthweight$};\fi
\ifnum\k=2\node[DynkinNode] (2) at (-2,0) {$\scriptscriptstyle\kthweight$};\fi
}
\draw[-Implies,double distance=1.5pt] (2) -- (1);
\draw (2) -- (1);
\end{tikzpicture}
}

\newcommand{\oneFDynkin}[1]{
\begin{tikzpicture}[scale=.6,anchor=base,baseline,inner sep=.15pt]
\foreach\kthweight[count=\k] in {#1}{
\ifnum\k=1\node[DynkinNode] (0) at (0,0) {$\scriptscriptstyle\kthweight$};\fi
\ifnum\k=2\node[DynkinNode] (1) at (1,0) {$\scriptscriptstyle\kthweight$};\fi
\ifnum\k=3\node[DynkinNode] (2) at (2,0) {$\scriptscriptstyle\kthweight$};\fi
\ifnum\k=4\node[DynkinNode] (3) at (3,0) {$\scriptscriptstyle\kthweight$};\fi
\ifnum\k=5\node[DynkinNode] (4) at (4,0) {$\scriptscriptstyle\kthweight$};\fi
}
\draw[thick] (0) -- (1) -- (2);
\draw[-Implies,double,thick] (2) -- (3);
\draw[thick] (3) -- (4);
\end{tikzpicture}
}

\newcommand{\FDynkin}[1]{
\begin{tikzpicture}[scale=.6,anchor=base,baseline,inner sep=.15pt]
\foreach\kthweight[count=\k] in {#1}{
\ifnum\k=1\node[DynkinNode] (1) at (-1,0) {$\scriptscriptstyle\kthweight$};\fi
\ifnum\k=2\node[DynkinNode] (4) at (-4,0) {$\scriptscriptstyle\kthweight$};\fi
\ifnum\k=3\node[DynkinNode] (2) at (-2,0) {$\scriptscriptstyle\kthweight$};\fi
\ifnum\k=4\node[DynkinNode] (3) at (-3,0) {$\scriptscriptstyle\kthweight$};\fi
}
\draw[thick] (1) -- (2);
\draw[-Implies,double,thick] (3) -- (2);
\draw[thick] (3) -- (4);
\end{tikzpicture}
}

\newcommand{\EDynkin}[1]{
\begin{tikzpicture}[scale=.4,anchor=base,baseline]
\foreach\kthweight[count=\k] in {#1}{
\pgfmathtruncatemacro{\prevnode}{\k-1}
\ifnum\k=1\node[DynkinNode] (\k) at (0,-1) {$\scriptscriptstyle\kthweight$};\fi 
\ifnum\k>1\node at (4-\k,0) {$\scriptscriptstyle\kthweight$};
\node[DynkinNode] (\k) at (4-\k,0) {$\scriptscriptstyle\kthweight$};
\ifnum\k>2\draw[thick] (\k) -- (\prevnode);\fi
\ifnum\k=4\draw[thick] (\k) -- (1);\fi
\fi
}
\end{tikzpicture}
}

\newcommand{\onEsixDynkin}[1]{
\begin{tikzpicture}[scale=.4,anchor=base,baseline]
\foreach\kthweight[count=\k] in {#1}{
\ifnum\k<6\node[DynkinNode] (\k) at (\k,0) {$\scriptscriptstyle\kthweight$};\fi
\ifnum\k=6\node[DynkinNode] (\k) at (3,-1) {$\scriptscriptstyle\kthweight$};\fi
\ifnum\k=7\node[DynkinNode] (\k) at (3,-2) {$\scriptscriptstyle\kthweight$};\fi
}
\draw[thick] (1) -- (2) -- (3) -- (4) -- (5);
\draw[thick] (3) -- (6) -- (7);
\end{tikzpicture}
}

\newcommand{\onEsevenDynkin}[1]{
\begin{tikzpicture}[scale=.4,anchor=base,baseline]
\foreach\kthweight[count=\k] in {#1}{
\ifnum\k<8\node[DynkinNode] (\k) at (\k,0) {$\scriptscriptstyle\kthweight$};\fi
\ifnum\k=8\node[DynkinNode] (\k) at (4,-1) {$\scriptscriptstyle\kthweight$};\fi
}
\draw[thick] (1) -- (2) -- (3) -- (4) -- (5) -- (6) -- (7);
\draw[thick] (4) -- (8);
\end{tikzpicture}
}

\newcommand{\onEightDynkin}[1]{
\begin{tikzpicture}[scale=.4,anchor=base,baseline]
\foreach\kthweight[count=\k] in {#1}{
\ifnum\k<9\node[DynkinNode] (\k) at (\k,0) {$\scriptscriptstyle\kthweight$};\fi
\ifnum\k=9\node[DynkinNode] (\k) at (6,-1) {$\scriptscriptstyle\kthweight$};\fi
}
\draw[thick] (1) -- (2) -- (3) -- (4) -- (5) -- (6) -- (7) -- (8);
\draw[thick] (6) -- (9);
\end{tikzpicture}
}

\newlength\dynshift

\let\ge\geqslant
\let\le\leqslant
\let\eps\varepsilon

\newcommand\mf\mathfrak
\newcommand\mr\mathrm
\newcommand\mb\mathbf
\newcommand\bb\mathbb
\newcommand\bi\boldsymbol
\newcommand\s{^{\mathrm s}}
\newcommand\n{^{\mathrm n}}
\newcommand\ev{_{\mathrm{ev}}}
\newcommand\od{_{\mathrm{odd}}}
\newcommand\pr{_{\mathrm{pr}}}
\newcommand\transpose{^\intercal}
\newcommand{\git}{/\!\!/}
\DeclareMathOperator\ad{ad}

\newtheorem{Theorem}{Theorem}[section]
\newtheorem{Conjecture}{Conjecture}[section]
\theoremstyle{definition}
\newtheorem{Definition}{Definition}[section]
\theoremstyle{remark}
\newtheorem{Example}{Example}[section]
\newtheorem{Remark}{Remark}[section]

\numberwithin{equation}{section}

\begin{document}

\title{Normal forms of nilpotent elements in semisimple Lie algebras}

\author{Mamuka Jibladze}
\address{Razmadze Mathematical Institute,
TSU, Tbilisi 0186, Georgia
}
\email{jib@rmi.ge}
\author{Victor G. Kac}
\address{Department of Mathematics, MIT,
77 Massachusetts Avenue, Cambridge, MA 02139, USA}
\email{kac@math.mit.edu}

\begin{abstract}
We find the normal form of nilpotent elements in semisimple
Lie algebras that generalizes the Jordan normal form
in $\mf{sl}_N$, using the theory of cyclic elements.

\

\

\

\begin{center}
\ \hfill\it\Large Dedicated to the memory of Tonny A. Springer
\end{center}
\end{abstract}

\maketitle

\tableofcontents

\section{Introduction}

Let $\mf g$ be a semisimple finite-dimensional Lie algebra over an algebraically closed field $\bb F$ of characteristic 0 and let $f$ be a non-zero nilpotent element of $\mf g$. By the Morozov-Jacobson theorem, the element $f$ can be included in an $\mf{sl}_2$-triple $\mf s=\{e,h,f\}$, unique, up to conjugacy by the centralizer of $f$ in the adjoint group $G$ \cite{K}, so that $[e,f]=h$, $[h,e]=2e$, $[h,f]=-2f$. Then the eigenspace decomposition of $\mf g$ with respect to $\ad h$ is a $\bb Z$-grading of $\mf g$:
\begin{equation}\label{eq:grading}
\mf g=\bigoplus_{j=-d}^d\mf g_j,\text{ where $\mf g_{\pm d}\ne0$.}
\end{equation}
The positive integer $d$ is called the \emph{depth} of the nilpotent element $f$ in $\mf g$. Choose a Cartan subalgebra $\mf h$ in $\mf g_0$ of $\mf g$ (it contains $h$ since $h$ is central in $\mf g_0$) and a subset of positive roots of $\mf g$,
compatible with the $\bb Z$-grading \eqref{eq:grading}, and let $\alpha_1$, ..., $\alpha_r$ be simple roots.
Then the Dynkin labels $\alpha_1(h)$, ..., $\alpha_r(h)$ take the values $0$, $1$, or $2$ only, and determine $f$ up to conjugation \cite{D}. Let $\mf g\ev=\bigoplus_{j\in\bb Z}\mf g_{2j}$.

An element of $\mf g$ of the form $f+E$, where $E$ is a non-zero element of $\mf g_d$, is called a \emph{cyclic element}, associated to $f$. In \cite{K} Kostant proved that any cyclic element, associated to a principal ($=$ regular) nilpotent element $f$, is regular semisimple, and in \cite{S} Springer proved that any cyclic element, associated to a subregular nilpotent element of a simple exceptional Lie algebra, is regular semisimple as well, and, moreover, found two more distinguished nilpotent elements in $\mr E_8$ with the same property.

A systematic study of cyclic elements began in the paper \cite{EKV}. Let us remind some terminology and results from it. A non-zero nilpotent element $f$ of $\mf g$ is called of \emph{nilpotent} (resp. \emph{semisimple}) \emph{type} if all cyclic elements, associated to $f$, are nilpotent (resp. there exists a semisimple cyclic element, associated to $f$). If neither of the above cases occurs, the element $f$ is called of \emph{mixed type}. The element $f$ is said to be of \emph{regular semisimple type} if there exists a regular semisimple cyclic element associated to $f$.

An important r\^ole in the study of cyclic elements, associated to a non-zero nilpotent element $f$, is played by the centralizer $\mf z(\mf s)$ in $\mf g$ of the $\mf{sl}_2$-triple $\mf s$ and by its centralizer $Z(\mf s)$ in $G$. Since $h\in\mf s$, the group $Z(\mf s)$ preserves the grading \eqref{eq:grading}, so that we have the linear algebraic groups, obtained by restricting the action of $Z(\mf s)$ to $\mf g_j$, which we denote by $Z(\mf s)|\mf g_j$, $j\in\bb Z$. The vector space $\mf g_j$ carries a $Z(\mf s)$-invariant non-degenerate bilinear form
\[
(x,y)=\varkappa((\ad f)^jx,y),\qquad x,y\in\mf g_j,
\]
where $\varkappa$ is the Killing form. The bilinear form $(\cdot,\cdot)$ is symmetric (resp. skew-symmetric) if $j$ is even (resp. odd) \cite{P}, \cite{EKV}.

The first main result of \cite{EKV} is the following theorem.

\begin{Theorem}
\begin{itemize}
\item[(a)] A non-zero nilpotent element $f$ is of nilpotent type if and only if the depth $d$ of $f$ is odd. In this case $Z(\mf s)|\mf g_d=\mr{Sp}(\mf g_d)$.
\item[(b)] A non-zero nilpotent element $f$ is of semisimple type if and only if the set
\begin{equation}
\mb S_{\mf g}(f)=\{ E\in\mf g_d\mid\text{ $f+E$ is semisimple} \}
\end{equation}
contains a non-empty Zariski open subset. The set $\mb S_{\mf g}(f)$ is a union of closed orbits of $Z(\mf s)$ in $\mf g_d$ and is conical.
\end{itemize}
\end{Theorem}

The set $\mb S_{\mf g}(f)\subseteq\mf g_d$ for all nilpotent elements $f$ of semisimple type in all semisimple Lie algebras $\mf g$ is explicitly described in \cite{EJK}.

The positive integer $\tilde d=d-1$ (resp. $=d$) if $f$ is of nilpotent (resp. semisimple or mixed) type in $\mf g$ is called the \emph{reduced depth} of $f$ in $\mf g$. Note that $\tilde d$ is the depth of $f$ in $\mf g\ev$. The dimension of the affine algebraic variety $\mf g_{\tilde d}\git Z(\mf s)$ is called the \emph{rank} of $f$ in $\mf g$. Obviously the rank of $f$ in $\mf g$ equals the rank of $f$ in $\mf g\ev$.

The key notion of the theory of cyclic elements is that of a \emph{reducing subalgebra} for a nilpotent element $f$ in $\mf g$.
It is a semisimple subalgebra $\mf q$ of $\mf g$, normalized by $\mf s$, such that $Z(\mf s)(\mf q\cap\mf g_{\tilde d})$ contains a non-empty Zariski open subset. Note that if $\mf q$ is a reducing subalgebra for $f$ in $\mf g$, then $\mf q\cap\mf g\ev$ is as well.
In \cite{EKV}  the following conjecture was proposed.

\begin{Conjecture}\label{conj:conj}
For any non-zero nilpotent element $f$ and $\mf{sl}_2$-triple $\mf s$ containing it, there is a unique minimal reducing subalgebra,
up to conjugacy by $Z(\mf s)$.
\end{Conjecture}
In the present paper we prove this conjecture (see Theorem \ref{thm:minred}).

The next important result of \cite{EKV} is the following theorem (which is stated here in a slightly different form).
\begin{Theorem}\label{thm:decomp}
Let $f$ be a non-zero nilpotent element of $\mf g$ of reduced depth $\tilde d$.
Then there exists a reducing subalgebra $\mf q\subseteq\mf g\ev$ for $f$, $f\in\mf s\subseteq\mf n(\mf q)=\mf q\oplus\mf z(\mf q)$,
where $\mf n(\mf q)$ (resp. $\mf z(\mf q)$) stands for the normalizer (resp. centralizer) of $\mf q$ in $\mf g$, such that the decomposition
\begin{equation}\label{eq:decomp}
\text{$f=f\s+f\n$, where $f\s\in\mf q$, $f\n\in\mf z(\mf q)$,}
\end{equation}
has the following properties:
\begin{itemize}
\item[(a)] $f\s$ is an element of semisimple type in $\mf q$.
\item[(b)] If (a) holds, then the reduced depth of $f\n$ in $[\mf z(\mf q),\mf z(\mf q)]$ is smaller than $\tilde d$.
\item[(c)] If $f$ is of semisimple type, $f$ is of regular semisimple type in $\mf q$.
\end{itemize}
\end{Theorem}

A collection of all nilpotent elements $f$ with the same $f\s$ in \eqref{eq:decomp}, is called the \emph{bush},
containing the nilpotent element $f\s$ of semisimple type.

A stronger version of Theorem \ref{thm:decomp} (c) was proved in \cite{EJK}:
\begin{Theorem}
If $f$ is a nilpotent element of semisimple type in $\mf g$, then there exists a reducing subalgebra $\mf q$ for $f$,
such that $\mf q\cap\mf g_d$ is a Cartan subspace of the linear algebraic group $Z(\mf s)|\mf g_d$.
\end{Theorem}

Here we should recall the notion of a \emph{polar} linear algebraic group $G|V$ and its Cartan subspace $C\subset V$, introduced in \cite{DK}.
It is a finite-dimensional faithful representation of a reductive algebraic group $G$ in a vector space $V$, which admits a subspace $C$,
called a \emph{Cartan subspace}, having the following properties, similar to that of a Cartan subalgebra in a simple Lie algebra:
\begin{itemize}
\item[(C1)] $G$-orbits of all $v\in C$ are closed,
\item[(C2)] any closed $G$-orbit in $V$ intersects $C$ by a non-empty finite subset,
\item[(C3)] all Cartan subspaces are conjugate by $G$.
\end{itemize}
We omit the formal definition of a polar representation; it is proved in \cite{DK} that it does satisfy properties (C1) -- (C3), which suffice for this paper.

It is proved in \cite{EJK} (by a case-wise verification) that all linear groups $Z(\mf s)|\mf g_{\tilde d}$ are polar.

The first main result of the paper is Theorem \ref{thm:minred} which proves Conjecture \ref{conj:conj}.
Namely, for a nilpotent element $f$ of reduced depth $\tilde d$,
let $C\subseteq\mf g_{\tilde d}$ be a Cartan subspace for the polar linear group $Z(\mf s)|\mf g_{\tilde d}$.
Let $\mf q(f,C)$ be the subalgebra of $\mf g$, generated by $f$ and $C$.
We show that $\mf q(f,C)$ is a minimal reducing subalgebra for $f$ if $f$ is of semisimple type,
and that $\mf q(f,C)=\mf q(f\s,C)\oplus\bb Cf\n$, where $f=f\s+f\n$ is the decomposition \eqref{eq:decomp}
for $f$ of nilpotent or mixed type, where $\mf q(f\s,C)$ is a minimal reducing subalgebra for $f$ ($f\n$ may be 0 if $f$ is of nilpotent type).

A nilpotent element $f$ of $\mf g$ is called \emph{irreducible} if its only reducing subalgebra is $\mf g$ itself.
All irreducible nilpotent elements in simple Lie algebras $\mf g$ are listed in \cite{EKV}*{Remark 5.4}.
They are listed in Table \ref{tab:irreds} below, reproduced from \cite{EJK}*{Table 1}. Note that $d=\tilde d$ is even for all of them.

\begin{table}[H]\renewcommand\thetable{1}\caption{Irreducible nilpotent elements in simple $\mf g$ (\cite{EJK}*{Table 1})\label{tab:irreds}}

\ \def\arraystretch{1.2}

\begin{tabular}{l|r|r|r|r}
$\mf g$&nilpotent $f$&depth&$\dim\mf g_d$&$Z(\mf s)|\mf g_d$\\
\hline\hline
$\mr A_{2k}$&$\mr A_{2k}$\hfill\ADynkin{2,2,{\,\cdots},2,2}&$4k$&1&{\bf1}\\
\hline
$\mr C_k$&$\mr C_k$\hfill\CDynkin{2,2,{\,\cdots},2}2&$4k-2$&1&{\bf1}\\
\hline
$\mr B_k$, $k\ne3$&$\mr B_k$\hfill\BDynkin{2,2,{\,\cdots},2}2&$4k-2$&1&{\bf1}\\
$\mr D_{2k+2}$&$\mr D_{2k+2}(a_k)$\hfill\DDynkin{2,0,2,{\,\cdots},2,0}22&$4k+2$&2&{\bf1} $\oplus$ {\bf1}\\
\hline
$\mr G_2$&$\mr G_2$\hfill\GDynkin{2,2}&10&1&{\bf1}\\
\hline
$\mr F_4$&$\mr F_4$\hfill\FDynkin{2,2,2,2}&22&1&{\bf1}\\
$\mr F_4$&$\mr F_4(a_2)$\hfill\FDynkin{2,0,0,2}&10&2&$\sigma_2\oplus{\bf1}$\\
\hline
$\mr E_6$&$\mr E_6(a_1)$\hfill\EDynkin{2,2,2,0,2,2}&16&1&{\bf1}\\
\hline
$\mr E_7$&$\mr E_7$\hfill\EDynkin{2,2,2,2,2,2,2}&34&1&{\bf1}\\
$\mr E_7$&$\mr E_7(a_1)$\hfill\EDynkin{2,2,2,0,2,2,2}&26&1&{\bf1}\\
$\mr E_7$&$\mr E_7(a_5)$\hfill\EDynkin{0,0,0,2,0,0,2}&10&3&$\sigma_3\oplus{\bf1}$\\
\hline
$\mr E_8$&$\mr E_8$\hfill\EDynkin{2,2,2,2,2,2,2,2}&58&1&{\bf1}\\
$\mr E_8$&$\mr E_8(a_1)$\hfill\EDynkin{2,2,2,0,2,2,2,2}&46&1&{\bf1}\\
$\mr E_8$&$\mr E_8(a_2)$\hfill\EDynkin{2,2,2,0,2,0,2,2}&38&1&{\bf1}\\
$\mr E_8$&$\mr E_8(a_4)$\hfill\EDynkin{0,2,0,2,0,2,0,2}&28&1&{\bf1}\\
$\mr E_8$&$\mr E_8(a_5)$\hfill\EDynkin{0,2,0,2,0,0,2,0}&22&2&$\sigma_2\oplus{\bf1}$\\
$\mr E_8$&$\mr E_8(a_6)$\hfill\EDynkin{0,0,0,2,0,0,2,0}&18&2&$\sigma_3$\\
$\mr E_8$&$\mr E_8(a_7)$\hfill\EDynkin{0,0,0,0,2,0,0,0}&10&4&$\sigma_5$
\end{tabular}

\end{table}

Here in all cases $k\ge1$, and $\sigma_n$ denotes the $n-1$-dimensional nontrivial irreducible representation of $S_n$.
Throughout the paper we use notation for nilpotent elements from \cite{CM};
in particular $\mr X_k$ stands for the principal nilpotent element in the simple Lie algebra of type $\mr X_k$.

Note that a nilpotent element $f$ in a semisimple Lie algebra $\mf g$ is irreducible if and only if
the projection $f_j$ of $f$ to each simple component $\mf g_j$ of $\mf g$ is irreducible in $\mf g_j$,
and the depth of $f_j$ in $\mf g_j$ is the same for all $j$.

An easy consequence of the theory of cyclic elements is Theorem \ref{thm:jordec},
which provides a \emph{normal form} of any non-zero nilpotent element $f$ of reduced depth $\tilde d$
in a semisimple Lie algebra $\mf g$ in the following sense.
There exist semisimple commuting subalgebras $\mf g[j]$, $j=1,...,s$, of $\mf g$,
and irreducible nilpotent elements $f[j]$ of reduced depth $\tilde d_j$ in $\mf g[j]$, such that:
\begin{itemize}
\item[(i)] $f=\sum_{j=1}^sf[j]$;
\item[(ii)] $\tilde d={\tilde d}_1>{\tilde d}_2>...>{\tilde d}_s$;
\item[(iii)]  the reduced depth and the rank of $f[j]$ in $\mf g[j]$ are the same as the reduced depth and the rank of $f-\sum_{k<j}f[k]$ in $\mf g(j-1)$ for all $j\ge1$, where $\mf g(j)$ is defined inductively as the derived subalgebra of the centralizer of $\mf g[j]$ in $\mf g(j-1)$, starting with $\mf g(0)=\mf g$.
\end{itemize}
Such a normal form exists for any non-zero nilpotent element $f$ and is unique up to conjugation by the centralizer of $f$ in $G$.

In the case $\mf g=\mf{sl}_N$ this normal form coincides with the Jordan normal form.
We list normal forms of all nilpotent elements of all simple Lie algebras $\mf g$
(the case of semisimple $\mf g$ obviously reduces to the simple ones).
The case of classical $\mf g$ is treated in Section \ref{sec:clas},
and normal forms of nilpotent elements in exceptional $\mf g$ are described in Tables 3--7 in Section \ref{sec:excep}.

In our next paper we use the normal form of a nilpotent element $f$
to construct a map from the set of nilpotent orbits in a simple Lie algebra $\mf g$
to the set of conjugacy classes $[w_f]$ in the Weyl group $W$ of $\mf g$.
It extends the construction of \cite{EKV}, where $[w_f]$ was constructed for $f$ of regular semisimple type
(based on the idea of \cite{K}, see also \cite{S}), to all $f$, using the normal form of $f$,
since all irreducible nilpotent elements in a simple Lie algebra are of regular semisimple type.
This map coincides with that in \cite{KL} for $f$ of regular semisimple type (in particular, for irreducible $f$, given by Table 8).
However it is different in general (cf. \cite{Spa}).

Throughout the paper the base field $\bb F$ is algebraically closed of characteristic 0.

\subsection*{Acknowledgements}
The first named author was partially supported by the grant FR-18-10849 of Shota Rustaveli National Science Foundation of Georgia.
The second named author was partially supported by the Bert and Ann Kostant fund, and by the Simons collaboration grant.

The authors would like to thank A.~Elashvili for very useful discussions. Extensive use of
the computer algebra system GAP, and the package SLA by Willem de Graaf in particular \cite{SLA}, is gratefully acknowledged.

\section{Minimal reducing subalgebras and the normal form of a nilpotent element}

\begin{Theorem}\label{thm:minred}
Let $f$ be a nilpotent element in $\mf g$ of reduced depth $\tilde d$.
Let $C\subseteq\mf g_{\tilde d}$ be a Cartan subspace of the $Z(\mf s)$-module $\mf g_{\tilde d}$.
Denote by $\mf q(f,C)$ the subalgebra of $\mf g$, generated by $f$ and $C$. Then
\begin{itemize}
\item[(a)] If $f$ is of semisimple type, then $\mf q(f,C)$ is a minimal reducing subalgebra for $f$.
\item[(b)] If $f$ is of nilpotent or mixed type, let $f=f\s+f\n$ be the decomposition \eqref{eq:decomp} with respect to a reducing subalgebra $\mf q$ for $f$.
Then $\mf q(f,C)=\mf q(f\s,C)\oplus\bb Cf\n$ (direct sum of ideals), and $\mf q(f\s,C)$ is a minimal reducing subalgebra for $f$.
\item[(c)] Any minimal reducing subalgebra for $f$ of semisimple (resp. nilpotent or mixed) type is obtained as in (a) (resp. (b)).
\item[(d)] Any two minimal reducing subalgebras for $f$ are conjugate by $Z(\mf s)$.
\end{itemize}
\end{Theorem}
\begin{proof}
(a) Let $\mf q$ be a minimal reducing subalgebra for $f$, with induced from \eqref{eq:grading} $\bb Z$-gradation $\mf q=\bigoplus_{j\in\bb Z}\mf q_j$.
By \cite{EJK}*{Theorem 4}, $\mf s\subseteq\mf q$, $Z(\mf s)|\mf q_{\tilde d}$ is a finite linear group,
and $\mf q_{\tilde d}$ is a Cartan subspace for $Z(\mf s)|\mf g_{\tilde d}$.
Since Cartan subspaces of $Z(\mf s)|\mf g_{\tilde d}$ are conjugate to each other, we may assume that $C=\mf q_{\tilde d}$.
Let $\tilde{\mf q}$ be the subalgebra of $\mf q$, generated by the $\mf s$-submodule $\mf m=\sum_{j\ge0}(\ad f)^j\mf q_{\tilde d}$.
$\tilde{\mf q}$ is semisimple by \cite{EKV}*{Proposition 3.10}, hence it is a reducing subalgebra for $f$ in $\mf q$.
Due to minimality of $\mf q$, $\tilde{\mf q}=\mf q$, and since $\mf q(f,C)$ contains $\mf m$,
we conclude that $\mf q(f,C)$ contains $\tilde{\mf q}$, hence coincides with $\mf q$.

(b) Let $\mf q$ be the reducing subalgebra for $f$ in $\mf g$ from Theorem \ref{thm:decomp}, and let $f=f\s+f\n$ be the decomposition \eqref{eq:decomp},
hence $f\s,f\n\in\mf g_{-2}$, $f\s\in\mf q$, $[f\s,f\n]=0$. Then obviously we have
\[
\mf q(f,C)\subseteq\mf q(f\s,C)\oplus\bb Cf.
\]
In order to prove the reverse inclusion, note that $\mf q(f,C)$ is an algebraic Lie algebra. By \cite{DSJKV}*{Theorem 3.15},
the element $f+E$, where $E$ is a non-zero element of $C$ with closed $Z(\mf s)$-orbit, is integrable,
which means that the nilpotent part of the Jordan decomposition of $f+E$ is $f\n$.
Hence $f\n\in\mf q(f,C)$, proving the reverse inclusion.

(c) and (d) follow since any two Cartan subspaces of the $Z(\mf s)$-module $\mf g_{\tilde d}$ are conjugate.
\end{proof}

\begin{Remark}
Let $f$ be a nilpotent element of even depth $d$ in a simple Lie algebra $\mf g$.
It was proved in \cite{DSJKV} (by a case-wise verification) that $Z(\mf s)|\mf g_{d-1}$ is a polar linear group as well.
One can show that an analogue of Theorem \ref{thm:minred} holds in this situation as well.
Namely, provided that there exists a non-zero Cartan subspace $C\subseteq\mf g_{d-1}$, the subalgebra $\mf q(f,C)$ is either semisimple,
or is a direct sum of semisimple subalgebra and a $1$-dimensional center spanned by a nilpotent element.
In the first (resp. second) case $f+E$ is semisimple for generic $E\in\mf g_{d-1}$ (resp. never semisimple).
\end{Remark}

Now it is easy to prove the existence and uniqueness of a normal form.

\begin{Theorem}\label{thm:jordec}
Let $f$ be a non-zero nilpotent element of reduced depth $\tilde d$ in a semisimple Lie algebra $\mf g$.
Then there exist semisimple commuting subalgebras $\mf g[j]$, $j=1,...,s$,
and irreducible nilpotent elements $f[j]$ of reduced depth $\tilde d_j$ in $\mf g[j]$, such that
\begin{itemize}
\item[(i)]
\begin{equation}\label{eq:jordecomp}
f=\sum_{j=1}^sf[j];
\end{equation}
\item[(ii)] $\tilde d=\tilde d_1>\tilde d_2>...>\tilde d_s$;
\item[(iii)] the reduced depth and the rank of $f[j]$ in $\mf g[j]$ are the same as the reduced depth and the rank of $f-\sum_{k<j}f[k]$ in $\mf g(j-1)$ for all $j\ge1$, where $\mf g(j)$ is defined inductively as the derived subalgebra of the centralizer of $\mf g[j]$ in $\mf g(j-1)$, starting with $\mf g(0)=\mf g$.
\end{itemize}
Two such normal forms are conjugate to each other by the centralizer of $f$ in $G$.
\end{Theorem}
\begin{proof}
Consider a minimal reducing subalgebra $\mf q$ of $f$ in $\mf g$, provided by Theorem \ref{thm:minred},
and the decomposition \eqref{eq:decomp}: $f=f\s+f\n$. Let $f[1]=f\s$, $\mf g[1]=\mf q$, $\tilde d_1=$ reduced depth of $f[1]$ in $\mf g[1]$.

Let $\mf g(1)$ be the derived Lie algebra of the centralizer of $\mf q$ in $\mf g$. Then $f\n\in\mf g(1)$ and we apply the same procedure as above
with $f$ replaced by $f\n$ and $\mf g$ replaced by $\mf g(1)$. Note that the reduced depth ${\tilde d}_2$ of $f\n$ in $\mf g(1)$ is strictly smaller than $\tilde d_1$ due to Theorem \ref{thm:decomp}(b). After finitely many such steps we obtain a normal form of $f$ in $\mf g$.

The uniqueness, up to conjugacy by the centralizer of $f$ in $G$, follows from conjugacy of $\mf{sl}_2$-triples, containing $f$,
and of minimal reducing subalgebras for $f$.
\end{proof}

\begin{Definition}
Decomposition \eqref{eq:jordecomp} is called the \emph{normal form} of the nilpotent element $f$ in the semisimple Lie algebra $\mf g$.
\end{Definition}

\section{ Normal forms of nilpotent elements in classical simple Lie algebras}\label{sec:clas}
\subsection{\texorpdfstring{$\mf g=\mf{sl}_N$, $N\ge2$}{g = sl(N), N > 1}}
Non-zero nilpotent elements $f$, up to conjugation, are parametrized by partitions of $N$
\begin{equation}\label{eq:partition}
\bi p=(p_1^{(r_1)},p_2^{(r_2)},...,p_s^{(r_s)}),\quad N=\sum_ir_ip_i
\end{equation}
where
\begin{equation}\label{eq:parcond}
p_1>...>p_s\ge1,\quad r_i\ge1.
\end{equation}
Then the associated to the partition $\bi p$ nilpotent element $f$ is of semisimple type if and only if
\begin{equation}\label{eq:ssa}
\bi p=(p_1^{(r_1)},1^{(r_2)}),
\end{equation}
and the bush, containing the nilpotent element, associated to this partition,
consists of all nilpotent elements, associated to partitions of $N$ with the same $p_1$ and $r_1$ \cite{EKV}*{Section 4}.
Hence the decomposition \eqref{eq:decomp} of $f$ can be described as follows.
Consider the following (regular) subalgebra of $\mf{sl}_N$:
\begin{equation}\label{eq:slss}
\underbrace{\mf{sl}_{p_1}\oplus\cdots\oplus\mf{sl}_{p_1}}_{\text{$r_1$ times}}\oplus\mf{sl}_{N_1}, \text{where $N_1=N-p_1r_1$},
\end{equation}
and denote by $f[j]$, $j=1,...,r_1$, the nilpotent Jordan block of size $p_1$ in the $j$-th copy of $\mf{sl}_{p_1}$.
Then
\begin{equation}\label{eq:sldecomp}
f\s=\sum_{j=1}^{r_1}f[j],\quad f\n=f-f\s.
\end{equation}
Consequently the first $r_1$ summands of the normal form of $f$ are $f[1]$, ..., $f[r_1]$, with $f[j]\in\mf g[j]$,
where $\mf g[j]$ is the $j$-th copy of $\mf{sl}_{p_1}$ in \eqref{eq:slss} if $p_1$ is odd,
and the subalgebra $\mf{sp}_{p_1}$ of $\mf{sl}_{p_1}$, containing $f[j]$, if $p_1$ is even.

Next, we apply the same procedure to the nilpotent element $f\n$ in $\mf{sl}_{N_1}$,
which is associated to the partition $(p_2^{(r_2)},...,p_s^{(r_s)})$ of $N_1$, etc.

We thus obtain the following
\begin{Theorem}
Let $\mf g=\mf{sl}_N$, and let $f=J_{n_1}\oplus...\oplus J_{n_r}$ be the Jordan normal form of $f$,
where $J_n$ is the nilpotent Jordan block of size $n$. Then the normal form of $f$ is as follows:
we let $f[j]=J_{n_j}\in\mf g[j]$, where $\mf g[j]$ is the subalgebra $\mf{sl}_{n_j}$ (resp. $\mf{sp}_{n_j}\subset\mf{sl}_{n_j}$),
containing $f[j]$, if $n_j$ is odd (resp. even).
\end{Theorem}
\subsection{\texorpdfstring{$\mf g=\mf{sp}_N$, $N\ge2$, even}{g = sp(2N), N > 0}}
Non-zero nilpotent elements $f$, up to conjugacy, are parametrized by partitions of $N$ of the form \eqref{eq:partition}, \eqref{eq:parcond},
where $r_i$ is even if $p_i$ is odd. Then the associated to the partition $\bi p$ nilpotent element $f$ is of semisimple type
if and only if $\bi p$ is of the form \eqref{eq:ssa}, and the bush, containing the nilpotent element, associated to this partition,
as for $\mf{sl}_N$, consists of all nilpotent elements, associated to partitions of $N$ with the same $p_1$ and $r_1$ \cite{EKV}*{Section 4}.

In order to describe the decomposition \eqref{eq:decomp} of $f$, consider two cases.

Case (a): $p_1$ is even. Consider the following subalgebra of $\mf{sp}_N$ (cf. \eqref{eq:slss})
\begin{equation}\label{eq:spss}
\underbrace{\mf{sp}_{p_1}\oplus\cdots\oplus\mf{sp}_{p_1}}_\text{$r_1$ times}\oplus\mf{sp}_{N_1}, \text{where $N_1=N-p_1r_1$,}
\end{equation}
and denote by $f[j]$, $j=1,...,r_1$, the nilpotent Jordan block of size $p_1$ in the $j$-th copy of $\mf{sp}_{p_1}$.
Then \eqref{eq:sldecomp} holds. Consequently the first $r_1$ summands of the normal form of $f$ are $f[1]$, ..., $f[r_1]$, with $f[j]\in\mf g[j]$,
where $\mf g[j]$ is the $j$-th copy of $\mf{sp}_{p_1}$ in \eqref{eq:spss}.

Case (b): $p_1$ is odd, then $r_1$ is even: $r_1=2k_1$. Consider the following subalgebra of $\mf{sp}_N$:
\begin{equation}\label{eq:spssb}
\underbrace{\mf{sl}_{p_1}\oplus\cdots\oplus\mf{sl}_{p_1}}_\text{$k_1$ times}\oplus\mf{sp}_{N_1}, \quad N_1=N-p_1r_1,
\end{equation}
where $\mf{sl}_{p_1}$ is the subalgebra of $\mf{sp}_{p_1}$ in \eqref{eq:spss} embedded via the map $a\mapsto\operatorname{diag}(a,-a\transpose)$.
This subalgebra contains the matrix $\tilde J_{p_1}=J_{p_1}\oplus(-J_{p_1}\transpose)$.
Denote by $f[j]$, $j=1,...,k_1$, the submatrix $\tilde J_{p_1}$ in the $j$-th copy of $\mf{sl}_{p_1}$. Then \eqref{eq:sldecomp} holds.
Consequently the first $k_1$ summands of the normal form of $f$ are $f[1]$, ..., $f[k_1]$ with $f[j]\in\mf g[j]$,
where $\mf g[j]$ is the $j$-th copy of $\mf{sl}_{k_1}$ in \eqref{eq:spssb}.

Next, we apply the same procedure to the nilpotent element $f\n$ in $\mf{sp}_{N_1}$,
associated to the partition $(p_2^{(r_2)},...,p_s^{(r_s)})$ of $N_1$, etc.

We thus obtain the following
\begin{Theorem}
Let $\mf g=\mf{sp}_N$, and let $f$ be a non-zero nilpotent element of $\mf g$.
Then $f$ is a direct sum of Jordan blocks $J_{n_j}$, where $n_j$ is even,
and the blocks $\tilde J_{n_j}=J_{n_j}\oplus(-J_{n_j}\transpose)$ if $n_j$ is odd.
The normal form of $f$ is then as follows: we let $f[j]=J_{n_j}\in\mf g[j]$,
where $\mf g[j]$ is the subalgebra $\mf{sp}_{n_j}$ containing $f[j]$, if $n_j$ is even,
and $f[j]=\tilde J_{n_j}\in\mf g[j]$, where $\mf g[j]$ is the subalgebra $\mf{sl}_{n_j}$, containing $f[j]$,
embedded in $\mf{sp}_{2n_j}$ via the map $a\mapsto(a,-a\transpose)$ if $n_j$ is odd.
\end{Theorem}

\subsection{\texorpdfstring{$\mf g=\mf{so}_N$, $N\ge7$}{g = so(N), N > 6}}
We may assume that $N\ge7$ since $\mf{so}_3=\mf{sp}_2$, $\mf{so}_4=2\mf{sp}_2$, $\mf{so}_5=\mf{sp}_4$, $\mf{so}_6=\mf{sl}_4$.
Non-zero nilpotent elements $f$, up to conjugation, are parametrized by partitions of $N$ of the form \eqref{eq:partition}, \eqref{eq:parcond},
where $r_i$ us even if $p_i$ is even. Such a partition is called \emph{orthogonal}.
The associated to the partition $\bi p$ nilpotent element $f$ is of semisimple type in the following cases \cite{EKV}:
\begin{itemize}
\item[(a)] $\bi p=(3,1^{(r_2)})$;
\item[(b)] $\bi p=(p_1,1^{(r_2)})$, $p_1\ge5$ odd;
\item[(c)] $\bi p=(p_1,p_1-2,1^{(r_3)})$, $p_1\ge5$ odd;
\item[(d)] $\bi p=(p_1^{(r_1)},1^{(r_2)})$, $p_1\ge2$ even, $r_1\ge2$ even;
\item[(e)] $\bi p=(p_1^{(r_1)},1^{(r_2)})$, $p_1\ge3$ odd, $r_1\ge2$ even.
\end{itemize}
The depth $d$ of $f$ is equal to $2p_1-4$ in cases (a) -- (c), and to $2p_1-2$ in cases (d), (e).

Bushes, containing the nilpotent elements of semisimple type, consist of nilpotent elements, associated to the following partitions of $N$
(with all even parts having even multiplicities) \cite{EKV}:
\begin{itemize}
\item[(a)] the partition itself;
\item[(b)] all partitions with the same $p_1$ and $r_1=1$, satisfying $p_2<p_1-2$;
\item[(c)] all partitions with the same $p_1$, $r_1=1$, and $p_2=p_1-2$;
\item[(d)] all partitions with the same $p_1$ and the same $r_1$;
\item[(e)] all partitions with the same $p_1$ with multiplicity $r_1$ or $r_1+1$.
\end{itemize}

It is easy to deduce from \cite{EKV} or \cite{EJK} the following

\begin{table}[H]\caption{Normal forms of nilpotent elements of semisimple type in $\mf{so}_N$}\label{tab:redorth}

\

\begin{tabular}{ll|l|l}
&$f$&normal form&embedding\\
\hline
(a)&$(3,1^{(r_2)})$&$2\mr C_1$&$\mf{so}_4\subset\mf{so}_N$\\
(b)$_7$&$(7,1^{(r_2)})$&$\mr G_2$&$\mr G_2\subset\mf{so}_7$\\
(b)&$(2k+1,1^{(r_2)})$, $k\ne1,3$&$\mr B_k$&$\mf{so}_{2k+1}\subset\mf{so}_N$\\
(c)&$(2k+3,2k+1,1^{(r_3)})$&$\mr D_{2k+2}(a_k)$&$\mf{so}_{4k+4}\subset\mf{so}_N$\\
(d)&$((2k)^{(2n)},1^{(r_2)})$&$n\mr C_k$&$n\,\mf{sp}_{2k}\subset\mf{so}_N$\\
(e)&$((2k+1)^{(2n)},1^{(r_2)})$&$n\mr A_{2k}$&$n\,\mf{sl}_{2k+1}\subset\mf{so}_N$
\end{tabular}
\end{table}

The embeddings are constructed as follows. We view $\mf{so}_N$ as the Lie algebra $\mf{so}(V)$,
where $V$ is an $N$-dimensional vector space with a non-degenerate symmetric bilinear form.

(a), (b), (c). Taking a subspace $U\subseteq V$ for which the restriction of the bilinear form is non-degenerate,
we obtain the embedding $\mf{so}(U)\subseteq\mf{so}(V)$;

(b)$_7$ $G_2\subset\mf{so}_7$ by taking the irreducible 7-dimensional representation of $G_2$;

(d) If we view $\mf{sp}_m$ as the Lie algebra $\mf{sp}(U)$, where $U$ is an $m$-dimensional vector space
with a non-degenerate skewsymmetric bilinear form, then $U\otimes\bb C^2$,
where $\bb C^2$ is endowed with a non-degenerate skew-symmetric bilinear form,
defines an embedding $\mf{sp}_m\subset\mf{so}_{2m}\subseteq\mf{so}_N$.

(e) If we view $\mf{sl}_m$ as the Lie algebra $\mf{sl}(U)$, then taking $U+U^*$ with the non-degenerate symmetric bilinear form
defined by letting $U$ and $U^*$ to be isotropic, defines an embedding $\mf{sl}_m\subset\mf{so}_{2m}\subseteq\mf{so}_N$.

Given a partition $\bi p$ of the form \eqref{eq:partition}, \eqref{eq:parcond} of $N$ with even number of even parts
we break it in a union of boxes $\bi b_1$, $\bi b_2$, ..., which are partitions of integers $\le N$ as follows:

Case 1. $p_1$ is odd $\ge3$, $r_1=1$, $p_2\ne p_1-1$.

($\alpha$) $p_2<p_1-2$, then we let $\bi b_1=(p_1)$

($\beta$) $p_2=p_1-2$, then we let $\bi b_1=(p_1,p_1-2)$

Case 2. $p_1$ is even $\ge2$, $r_1\ge2$ is even. Then we let $\bi b_1=(p_1^{(r_1)})$.

Case 3. $p_1$ is odd $\ge3$, $r_1\ge2$.

($\alpha$) $r_1$ is even, then we let $\bi b_1=(p_1^{(r_1)})$

($\beta$) $r_1$ is odd, then we let $\bi b_1=(p_1^{(r_1-1)})$

Case 4. $p_1$ is odd $\ge3$, $r_1=1$, $p_2=p_1-1$. Then the associated nilpotent element $f$ is of nilpotent type,
and we let $\bi p=\bi p\ev\coprod\bi p\od$, where $\bi p\ev$ (resp. $\bi p\od$) is the subpartition of the partition $\bi p$,
consisting of its even (resp. odd) parts. Obviously, both are proper subpartitions, which correspond to even nilpotent elements
of mixed type in the subalgebras $\mf{so}_{|\bi p\ev|}$ (resp. $\mf{so}_{|\bi p\od|}$) of $\mf{so}_N$.

Next, we remove the box $\bi b_1$ from the partition $\bi p$ and continue this procedure in cases 1--3;
in Case 4 we apply this procedure to the partitions $\bi p\ev$ and $\bi p\od$.

Alternatively, we may begin by splitting the partition $\bi p$ in the disjoint union of $\bi p\ev$ and $\bi p\od$
and apply the above procedure separately to these subpartitions. Then Case 4 does not occur.

We thus obtain a collection of subalgebras $\tilde{\mf g}[j]$ of $\mf{so}_N$ and nilpotent elements $f[j]$ of semisimple type in these subalgebras,
and we replace $\tilde{\mf g}[j]$ by its subalgebra $\mf g[j]$, in which $f[j]$ is an irreducible nilpotent element, using Table \ref{tab:redorth}.

The following theorem is now clear.

\begin{Theorem}
The above procedure defines a normal form of any non-zero nilpotent element $f$ in $\mf{so}_N$, $N\ge7$.
\end{Theorem}\qed

\begin{Example}
Consider the nilpotent orbit for type $\mr A$ corresponding to the partition
\begin{equation}\label{eq:exsl}
(24^{(3)},23^{(4)},21^{(5)},18,13^{(4)},10^{(5)},8^{(6)},3^{(2)},2,1^{(5)}).
\end{equation}
From the even parts $(24^{(3)},18,10^{(5)},8^{(6)},2)$ one obtains the decomposition $3\mr C_{12}+\mr C_9+5\mr C_5+6\mr C_4+\mr C_1$,
and the odd parts $(23^{(4)},21^{(5)},13^{(4)},3^{(2)},1^{(5)})$ give $4\mr A_{22}+5\mr A_{20}+4\mr A_{12}+2\mr A_2$.
The normal form of \eqref{eq:exsl} is then the sum of these two sums, arranged in decreasing order of depths:
\[
3\mr C_{12}+3\mr A_{22}+5\mr A_{20}+\mr C_9+3\mr A_{12}+\mr C_5+\mr C_4+2\mr A_2+\mr C_1.
\]

For type $\mr C$ take, for example, the partition
\begin{equation}\label{eq:exsp}
(19^{(8)},17^{(4)},12^{(6)},11^{(10)},10^{(3)},6,5^{(4)},2^{(7)},1^{(2)}).
\end{equation}
The even parts $(12^{(6)},10^{(3)},6,2^{(7)})$ of this partition produce the decomposition $6\mr C_6+3\mr C_5+\mr C_3+7\mr C_1$,
and the odd parts $(19^{(8)},17^{(4)},11^{(10)},5^{(4)},1^{(2)})$ give rise to $4\mr A_{18}+2\mr A_{16}+5\mr A_{10}+2\mr A_4$.
The normal form of \eqref{eq:exsp} is again the sum of these two, arranged in decreasing order of depths:
\[
4\mr A_{18}+2\mr A_{16}+6\mr C_6+5\mr A_{10}+3\mr C_5+\mr C_3+2\mr A_4+7\mr C_1.
\]

For an orthogonal example consider the partition
\begin{equation}\label{eq:exso}
(20^{(4)},17^{(5)},15^{(6)},13^{(4)},10^{(2)},9^{(4)},8^{(2)},7^{(3)},5^{(4)},4^{(4)},3^{(8)},2^{(8)},1^{(6)}).
\end{equation}
Here the even parts $(20^{(4)},10^{(2)},8^{(2)},4^{(4)},2^{(8)})$ give the decomposition $2\mr C_{10}+\mr C_5+\mr C_4+2\mr C_2+4\mr C_1$,
whereas the odd parts $(17^{(5)},15^{(6)},13^{(4)},9^{(4)},7^{(3)},5^{(4)},3^{(8)},1^{(6)})$ become subdivided into blocks
\[
(17^{(4)}\mid17,15\mid15^{(4)}\mid15,13\mid13^{(2)}\mid13\mid9^{(4)}\mid7^{(2)}\mid7,5\mid5^{(2)}\mid5,3\mid3^{(6)}\mid3,1^{(6)}),
\]
and this gives $2\mr A_{16}+\mr D_{16}(a_7)+2\mr A_{14}+\mr D_{14}(a_6)+\mr A_{12}+\mr B_6+2\mr A_8+\mr A_6+\mr D_6(a_2)+\mr A_4+\mr D_4(a_1)+3\mr A_2+2\mr C_1$.
Also here, the normal form of \eqref{eq:exso} is the sum of these two sums, arranged in decreasing order of depths:
\[
2\mr C_{10}+2\mr A_{16}+\mr D_{16}(a_7)+2\mr A_{14}+\mr D_{14}(a_6)+\mr A_{12}+\mr B_6+\mr C_5+2\mr A_8+\mr C_4+\mr A_6+\mr D_6(a_2)+\mr A_4+(2\mr C_2
+\mr D_4(a_1))+3\mr A_2+6\mr C_1
\]
\end{Example}

\section{Normal forms of nilpotent elements in exceptional Lie algebras}\label{sec:excep}

We will use the following embeddings of subalgebras for the normal forms.

\ \setlength\tabcolsep{2pt}

\begin{tabular}{rlll}
I&\multicolumn{3}{l}{Regular subalgebra (i.~e. normalized by a Cartan subalgebra).}\\
II&\multicolumn{3}{l}{Folding of the Dynkin diagram:}\\
&(1)$_n$&$\mr B_n\subset\mr A_{2n}$\\
&(2)$_n$&$\mr C_n\subset\mr A_{2n-1}$\\
&(3)$_n$&$\mr B_n\subset\mr D_{n+1}$\\
&(4)&$\mr G_2\subset\mr D_4$\\
&(5)&$\mr F_4\subset\mr E_6$\\
III&\multicolumn{3}{l}{Restriction:}\\
&(1)&$\mr G_2\subset\mr B_3$; $e_1,e_2+e_3$ are root vectors&attached to simple roots of $\mr G_2$,\\
&&\qquad where $e_1,e_2,e_3$ are root vectors,&attached to simple roots of $\mr B_3$.\\
&(2)$_{n_1,...,n_t}$&$\mf{so}_{n_1}\oplus...\oplus\mf{so}_{n_t}\subset\mf{so}_{n_1+\cdots+n_t+\cdots}$\\
&(3)$_n$&$\mf{sl}_n\subset\mf{sp}_{2n}$\\
&(4)$_n$&$\mf{sl}_n\subset\mf{so}_{2n}$\\
IV&\multicolumn{3}{l}{Centralizer of a simple subalgebra:}\\
&(IV)$_{\mr B_3\subset\mr D_4}$,&where $\mr B_3$ is folding of $\mr D_4$,&$\mf z(\mr B_3)\cong\mr C_1$\\
&(IV)$_{\mr B_4\subset\mr E_7}$,&where $\mr B_4$ is folding of $\mr D_5\subset\mr E_7$,&$\mf z(\mr B_4)\cong\mr C_1\oplus\mr C_1'$,\\
&&\multicolumn{2}{l}{\qquad with $\mr C_1$ (resp. $\mr C_1'$) a regular (resp. not regular) subalgebra of $\mr E_7$}\\
&(IV)$_{\mr F_4\subset\mr E_7}$,&where $\mr F_4$ is folding of $\mr E_6\subset\mr E_7$,&$\mf z(\mr F_4)\cong\mr C_1$\\
&(IV)$_{\mr B_4\subset\mr E_8}$,&where $\mr B_4$ is folding of $\mr D_5\subset\mr E_8$,&$\mf z(\mr B_4)\cong\mr B_3$,\\
&&\multicolumn{2}{l}{\qquad with $\mr B_2$ and $\mr C_1+\tilde{\mr C}_1$ regular subalgebras in $\mr B_3$}\\
&(IV)$_{\mr B_3\subset\mr E_8}$,&where $\mr B_3$ is folding of $\mr D_4\subset\mr E_8$,&$\mf z(\mr B_3)\cong\mr B_4$\\
&(IV)$_{\mr B_5\subset\mr E_8}$,&where $\mr B_5$ is folding of $\mr D_6\subset\mr E_8$,&$\mf z(\mr B_5)\cong\mr B_2=\mr C_2$\\
&(IV)$_{\mr F_4\subset\mr E_8}$,&where $\mr F_4$ is folding of $\mr E_6\subset\mr E_8$,&$\mf z(\mr F_4)\cong\mr G_2$, $\mf z(\mr G_2)\cong\mr F_4$.
\end{tabular}

\

In Tables 3--7 below we list all nilpotent elements $f$, up to conjugacy, in exceptional Lie algebras $\mf g$,
in the increasing order of their depths $d$, their representatives in $\mf g$, their normal forms $f=\sum_jf[j]$,
and describe embeddings $\mf g[j]\subseteq\mf g$.
In each block of the tables we list first the element of semisimple (resp. nilpotent) type, and after that all elements
in the same bush with the element of semisimple type.
The normal forms of these elements are obtained by adding to the normal form of the element of semisimple type the summands given for them in the tables.

It turns out that for all exceptional $\mf g$ the bush of $f$ of nilpotent type in $\mf g$ contains $f$ only,
namely, $f$ is a nilpotent element of semisimple type in $\mf g\ev$.

In all tables $f\pr$ denotes the sum of all root vectors, attached to negative simple roots.
An element $f_{k_1...k_r}$ denotes a root vector, attached to the negative root $-\sum_ik_i\alpha_i$.
The representatives of $f$ in the Tables are given in terms of these root vectors.
In most of the cases the roots occurring there are linearly independent, and we can take all coefficients equal 1.
However, in a few cases, namely, for nilpotent elements $\mr D_5(a_1)$ in $\mr E_6$, $\mr D_6(a_1)$, $\mr E_7(a_4)$, $\mr E_7(a_2)$ in $\mr E_7$,
and $\mr D_5(a_1)+\mr A_2$, $\mr D_7(a_2)$, $\mr E_7(a_4)$, $\mr E_7(a_2)$ in $\mr E_8$, the roots occurring there are linearly dependent;
in these cases we use the choice of root vectors from \cite{SLA}.

These tables are deduced from Tables 5.1 -- 5.4 of \cite{EKV}. Again, we use notation for nilpotent elements used in \cite{CM}.
However, in some cases the more adequate notation of \cite{D} is given in parentheses.

The column with the heading ``representative'' contains a linear combination of negative root vectors that belongs to the required nilpotent orbit.
Linear combinations grouped in square brackets represent the separate summands of type $f[j]$ irreducible in a subalgebra $\mf g[j]$ according to normal form. Linear combinations grouped in parentheses represent single root vectors in non-regular semisimple subalgebras.

The column with the heading ``normal form'' lists types of the summands $f[j]$ as occurring in the Table \ref{tab:irreds} of irreducible nilpotent elements.

For example, normal form $\mr D_4(a_1)+2\mr C_1$ of the nilpotent element $f$ of type $\mr A_3+\mr A_2$ in
$\mf g=\mr E_7$ (in Table \ref{tab:E7}) means that $\mf g[1]=\mr D_4$, $f[1]$ is of type $\mr D_4(a_1)$ (fourth row of Table \ref{tab:irreds} with $k=1$), while both $f[2]$ in $\mf g[2]$ and $f[3]$ in $\mf g[3]$ are of type $\mr C_1$ (second row of Table \ref{tab:irreds} with $k=1$).

For another example, the nilpotent with label $\mr A_4+\mr A_3$ in $\mr E_8$ (Table \ref{tab:E8}) with representative $[f_{\foreight{.5}00000001}+f_{\foreight{.5}00000010}+f_{\foreight{.5}00000100}+f_{\foreight{.5}00001000}]
+[(f_{\foreight{.5}12354321}+f_{\foreight{.5}00100000})+f_{\foreight{.5}10000000}]$ has normal form $\mr A_4+\mr C_2$; here, the first sum in square brackets is the principal nilpotent element represented by the sum of negative simple root vectors of the regular simple subalgebra of type $\mr A_4$, while in the second square brackets we have the principal nilpotent element in a non-regular simple subalgebra of type $\mr C_2$, where $(f_{\foreight{.5}12354321}+f_{\foreight{.5}00100000})$ is a short negative simple root vector of this subalgebra and $f_{\foreight{.5}10000000}$ is its long negative simple root vector.

\begin{table}[H]\caption{Normal forms of nilpotent elements in $\mr G_2$}

\ \def\arraystretch{1.2}

\begin{tabular}{l|c|l|c|c}
$f$&$d$&representative&normal form&embedding\\
\hline
\hline
$\mr A_1$&2&$f_{10}$&$\mr C_1$&regular\\
\hline
$\tilde{\mr A}_1$&3&$f_{01}$&$\tilde{\mr C}_1$&regular\\
\hline
$\mr G_2(a_1)$&4&$f_{10}+f_{13}$&$\mr A_2$&regular\\
\hline
$\mr G_2$&10&$f\pr$&$\mr G_2$&regular
\end{tabular}
\end{table}

\begin{table}[H]\caption{Normal forms of nilpotent elements in $\mr F_4$}


\ \def\arraystretch{1.2}

\begin{tabular}{l|c|l|l|l}
$f$&$d$&representative&normal form&embedding\\
\hline
\hline
$\mr A_1$&2&$f_{0100}$&$\mr C_1$&regular\\
\hline
$\tilde{\mr A}_1$&2&$f_{0100}+f_{0120}$&$2\mr C_1$&regular\\
\hline
$\mr A_1+\tilde{\mr A}_1$&3&$f_{0100}+f_{0120}+f_{0122}$&$3\mr C_1$&regular\\
\hline
$\mr A_2$&4&$f_{1000}+f_{0100}$&$\mr A_2$&regular\\
$\mr A_2+\tilde{\mr A}_1$&&\hfill${}+f_{0001}$&\hfill${}+\tilde{\mr C}_1$&regular\\
\hline
$\tilde{\mr A}_2$&4&$f_{0010}+f_{0001}$&$\tilde{\mr A}_2$&regular\\[\dynshift]
\hline
$\tilde{\mr A}_2+\mr A_1$&5&$f_{0010}+f_{0001}$\hfill${}+f_{1000}$&$\tilde{\mr A}_2$\hfill${}+\mr C_1$&regular\\
\hline
$\mr B_2$&6&$f_{0100}+f_{0010}$&$\mr B_2$&regular\\
$\mr C_3(a_1)$\hfill(=$\mr B_2+\mr A_1$)&&\hfill${}+f_{0111}$&\hfill${}+\mr C_1$&regular\\
\hline
$\mr F_4(a_3)$&6&$f_{1000}+f_{0100}+f_{0120}+f_{0122}$&$\mr D_4(a_1)$&regular\\
\hline
$\mr B_3$&10&$f_{1000}+(f_{0100}+f_{0010})$&$\mr G_2$&III(1)\\
\hline
$\mr C_3$&10&$f_{0001}+f_{0010}+f_{0100}$&$\mr C_3$&regular\\
\hline
$\mr F_4(a_2)$&10&$f_{0100}+f_{0120}+f_{1110}+f_{0001}$&$\mr F_4(a_2)$&regular\\
\hline
$\mr F_4(a_1)$\hfill($=\mr B_4$)&14&$f_{0100}+f_{0120}+f_{1000}+f_{0001}$&$\mr B_4$&regular\\
\hline
$\mr F_4$&22&$f\pr$&$\mr F_4$&regular
\end{tabular}
\end{table}

\begin{table}[H]\caption{Normal forms of nilpotent elements in $\mr E_6$}

\ \def\arraystretch{1.2}

\setlength\dynshift{.25em}

\begin{tabular}{l|c|l|l|l}
$f$&$d$&representative&normal form&embedding\\
\hline
\hline
$\mr A_1$&2&$f_{\foresix{.1}{.5}000001}$&$\mr C_1$&regular\\[\dynshift]
\hline
$2\mr A_1$&2&$f_{\foresix{.1}{.5}000001}+f_{\foresix{.1}{.5}100000}$&$2\mr C_1$&regular\\[\dynshift]
\hline
$3\mr A_1$&3&$f_{\foresix{.1}{.5}000001}+f_{\foresix{.1}{.5}000100}+f_{\foresix{.1}{.5}100000}$&$3\mr C_1$&regular\\[\dynshift]
\hline
$\mr A_2$&4&$f_{\foresix{.1}{.5}000100}+f_{\foresix{.1}{.5}010000}$&$\mr A_2$&regular\\
$\mr A_2+\mr A_1$&&\hfill${}+f_{\foresix{.1}{.5}100000}$&\hfill${}+\mr C_1$&regular\\
$\mr A_2+2\mr A_1$&&\hfill${}+f_{\foresix{.1}{.5}000001}+f_{\foresix{.1}{.5}100000}$&\hfill${}+2\mr C_1$&regular\\[\dynshift]
\hline
$2\mr A_2$&4&$[f_{\foresix{.1}{.5}000001}+f_{\foresix{.1}{.5}100000}]+[f_{\foresix{.1}{.5}000010}+f_{\foresix{.1}{.5}001000}]$&$2\mr A_2$&regular\\[\dynshift]
\hline
$2\mr A_2+\mr A_1$&5&$[f_{\foresix{.1}{.5}000001}+f_{\foresix{.1}{.5}100000}]+[f_{\foresix{.1}{.5}000010}+f_{\foresix{.1}{.5}001000}]$\hfill${}+f_{\foresix{.1}{.5}010000}$ &$2\mr A_2$\hfill${}+\mr C_1$&regular\\[\dynshift]
\hline
$\mr A_3$&6&$f_{\foresix{.1}{.5}000001}+(f_{\foresix{.1}{.5}000010}+f_{\foresix{.1}{.5}000100})$&$\mr C_2$&folding of $\mr A_3$\\
$\mr A_3+\mr A_1$&&\hfill${}+f_{\foresix{.1}{.5}100000}$&\hfill${}+\mr C_1$&\hfill+regular\\[\dynshift]
\hline
$\mr D_4(a_1)$&6&$f_{\foresix{.1}{.5}000010}+f_{\foresix{.1}{.5}010000}+f_{\foresix{.1}{.5}000100}+f_{\foresix{.1}{.5}001110}$&$\mr D_4(a_1)$&regular\\[\dynshift]
\hline
$\mr A_4$&8&$f_{\foresix{.1}{.5}000001}+f_{\foresix{.1}{.5}000010}+f_{\foresix{.1}{.5}000100}+f_{\foresix{.1}{.5}010000}$&$\mr A_4$&regular\\
$\mr A_4+\mr A_1$&&\hfill${}+f_{\foresix{.1}{.5}100000}$&\hfill${}+\mr C_1$&regular\\[\dynshift]
\hline
$\mr D_4$&10&$(f_{\foresix{.1}{.5}000010}+f_{\foresix{.1}{.5}010000}+f_{\foresix{.1}{.5}001000})+f_{\foresix{.1}{.5}000100}$&$\mr G_2$&folding of $\mr D_4$\\
$\mr D_5(a_1)$&&\hfill${}+(f_{\foresix{.1}{.5}010111}+f_{\foresix{.1}{.5}001111})$&\hfill${}+\mr C_1$&\hfill+(IV)$_{\mr B_3\subset\mr D_4}$\\[\dynshift]
\hline
$\mr A_5$&10 &$(f_{\foresix{.1}{.5}000001}+f_{\foresix{.1}{.5}100000})+(f_{\foresix{.1}{.5}000010}+f_{\foresix{.1}{.5}001000})+f_{\foresix{.1}{.5}000100}$
&$\mr C_3$&folding of $\mr A_5$\\[\dynshift]
\hline
$\mr E_6(a_3)$\hfill($=\mr A_5+\mr A_1$)&10&$(f_{\foresix{.1}{.5}000001}+f_{\foresix{.1}{.5}100000})+(f_{\foresix{.1}{.5}000010}+f_{\foresix{.1}{.5}001000}) +f_{\foresix{.1}{.5}000100}+f_{\foresix{.1}{.5}111111}$&$\mr F_4(a_2)$&folding of $\mr E_6$\\[\dynshift]
\hline
$\mr D_5$&14&$f_{\foresix{.1}{.5}000001}+f_{\foresix{.1}{.5}000010}+f_{\foresix{.1}{.5}000100}+(f_{\foresix{.1}{.5}001000} +f_{\foresix{.1}{.5}010000})$&$\mr B_4$&folding of $\mr D_5$\\[\dynshift]
\hline
$\mr E_6(a_1)$&16&$f_{\foresix{.1}{.5}000001}+f_{\foresix{.1}{.5}000010}+f_{\foresix{.1}{.5}001000}+f_{\foresix{.1}{.5}100000} +f_{\foresix{.1}{.5}010100}+f_{\foresix{.1}{.5}000110}$&$\mr E_6(a_1)$&regular\\[\dynshift]
\hline
$\mr E_6$&22&$(f_{\foresix{.1}{.5}000001}+f_{\foresix{.1}{.5}100000})+(f_{\foresix{.1}{.5}000010}+f_{\foresix{.1}{.5}001000}) +f_{\foresix{.1}{.5}000100} +f_{\foresix{.1}{.5}010000}$&$\mr F_4$&folding of $\mr E_6$
\end{tabular}
\end{table}

\begin{table}[H]\caption{Normal forms of nilpotent elements in $\mr E_7$\label{tab:E7}}

\ \def\arraystretch{1.3}

\setlength\dynshift{.25em}

\resizebox{.95\textwidth}{!}{%
\begin{tabular}{l|c|l|l|l}
$f$&$d$&representative&normal form&embedding\\
\hline
\hline
$\mr A_1$&2&$f_{\foreseven{.1}{.5}0000001}$&$\mr C_1$&regular\\[\dynshift]
\hline
$2\mr A_1$&2&$f_{\foreseven{.1}{.5}0000001}+f_{\foreseven{.1}{.5}0000100}$&$2\mr C_1$&regular\\[\dynshift]
\hline
$(3\mr A_1)''$&2&$f_{\foreseven{.1}{.5}0000001}+f_{\foreseven{.1}{.5}0000100}+f_{\foreseven{.1}{.5}0100000}$&$(3\mr C_1)''$&regular\\[\dynshift]
\hline
$(3\mr A_1)'$&3&$f_{\foreseven{.1}{.5}0000001}+f_{\foreseven{.1}{.5}0000100}+f_{\foreseven{.1}{.5}0010000}$&$(3\mr C_1)'$&regular\\[\dynshift]
\hline
$4\mr A_1$&3&$f_{\foreseven{.1}{.5}0000001}+f_{\foreseven{.1}{.5}0000100}+f_{\foreseven{.1}{.5}0010000}+f_{\foreseven{.1}{.5}0100000}$&$4\mr C_1$&regular\\[\dynshift]
\hline
$\mr A_2$&4&$f_{\foreseven{.1}{.5}0010000}+f_{\foreseven{.1}{.5}1000000}$&$\mr A_2$&regular\\
$\mr A_2+\mr A_1$&&\hfill${}+f_{\foreseven{.1}{.5}0000001}$&\hfill${}+\mr C_1$&regular\\
$\mr A_2+2\mr A_1$&&\hfill${}+f_{\foreseven{.1}{.5}0000100}+f_{\foreseven{.1}{.5}0000001}$&\hfill${}+2\mr C_1$&regular\\
$\mr A_2+3\mr A_1$&&\hfill${}+f_{\foreseven{.1}{.5}0100000}+f_{\foreseven{.1}{.5}0000100}+f_{\foreseven{.1}{.5}0000001}$&\hfill${}+3\mr C_1$&regular\\[\dynshift]
\hline
$2\mr A_2$&4&$[f_{\foreseven{.1}{.5}0000001}+f_{\foreseven{.1}{.5}0000010}] +[f_{\foreseven{.1}{.5}0001000}+f_{\foreseven{.1}{.5}0100000}]$&$2\mr A_2$&regular\\[\dynshift]
\hline
$2\mr A_2+\mr A_1$&5&$[f_{\foreseven{.1}{.5}0000001}+f_{\foreseven{.1}{.5}0000010}] +[f_{\foreseven{.1}{.5}0001000}+f_{\foreseven{.1}{.5}0100000}]$\hfill${}+f_{\foreseven{.1}{.5}1000000}$&$2\mr A_2$\hfill${}+\mr C_1$&regular\\[\dynshift]
\hline
$\mr A_3$&6&$(f_{\foreseven{.1}{.5}0000001}+f_{\foreseven{.1}{.5}0000100})+f_{\foreseven{.1}{.5}0000010}$&$\mr C_2$&folding of $\mr A_3$\\
$(\mr A_3+\mr A_1)'$&&\hfill${}+f_{\foreseven{.1}{.5}0010000}$&\hfill${}+\mr C_1$&\hfill+regular\\
$(\mr A_3+\mr A_1)''$&&\hfill${}+f_{\foreseven{.1}{.5}0100000}$&\hfill${}+\mr C_1$&\hfill+regular\\
$\mr A_3+2\mr A_1$&&\hfill${}+f_{\foreseven{.1}{.5}0100000}+f_{\foreseven{.1}{.5}0010000}$&\hfill${}+2\mr C_1$&\hfill+regular\\[\dynshift]
\hline
$\mr D_4(a_1)$&6&$f_{\foreseven{.1}{.5}0000100}+f_{\foreseven{.1}{.5}0011000}+f_{\foreseven{.1}{.5}0100000}+f_{\foreseven{.1}{.5}0001100}$&$\mr D_4(a_1)$&regular\\
$\mr D_4(a_1)+\mr A_1$&&\hfill${}+f_{\foreseven{.1}{.5}0000001}$&\hfill${}+\mr C_1$&regular\\
$\mr A_3+\mr A_2$\hfill$(=\mr D_4(a_1)+2\mr A_1)$&&\hfill${}+f_{\foreseven{.1}{.5}0112221}+f_{\foreseven{.1}{.5}0000001}$&\hfill${}+2\mr C_1$&regular\\
$\mr A_3+\mr A_2+\mr A_1(=\mr D_4(a_1)+3\mr A_1)$&&\hfill${}+f_{\foreseven{.1}{.5}2234321}+f_{\foreseven{.1}{.5}0112221}+f_{\foreseven{.1}{.5}0000001}$
&\hfill${}+(3\mr C_1)''$&regular\\[\dynshift]
\hline
$\mr A_4$&8&$f_{\foreseven{.1}{.5}1000000}+f_{\foreseven{.1}{.5}0010000}+f_{\foreseven{.1}{.5}0001000}+f_{\foreseven{.1}{.5}0100000}$&$\mr A_4$&regular\\
$\mr A_4+\mr A_1$&&\hfill${}+f_{\foreseven{.1}{.5}0000001}$&\hfill${}+\mr C_1$&regular\\
$\mr A_4+\mr A_2$&&\hfill${}+f_{\foreseven{.1}{.5}0000010}+f_{\foreseven{.1}{.5}0000001}$&\hfill${}+\mr A_2$&regular\\[\dynshift]
\hline
$\mr D_4$&10&$(f_{\foreseven{.1}{.5}0000100}+f_{\foreseven{.1}{.5}0010000}+f_{\foreseven{.1}{.5}0100000})+f_{\foreseven{.1}{.5}0001000}$&$\mr G_2$&folding of $\mr D_4$\\
$\mr D_4+\mr A_1$&&\hfill${}+f_{\foreseven{.1}{.5}0000001}$&\hfill${}+\mr C_1$&\hfill+regular\\
$\mr D_5(a_1)$\hfill$(=\mr D_4+2\mr A_1)$&&\hfill${}+f_{\foreseven{.1}{.5}0112221}+f_{\foreseven{.1}{.5}0000001}$&\hfill${}+2\mr C_1$&\hfill+regular\\
$\mr D_5(a_1)+\mr A_1$\hfill$(=\mr D_4+3\mr A_1)$&&\hfill${}+f_{\foreseven{.1}{.5}2234321}+f_{\foreseven{.1}{.5}0112221}+f_{\foreseven{.1}{.5}0000001}$&\hfill${}+(3\mr C_1)''$&\hfill+regular\\[\dynshift]
\hline
$\mr A_5'$&10&$(f_{\foreseven{.1}{.5}0000001}+f_{\foreseven{.1}{.5}0010000}) +(f_{\foreseven{.1}{.5}0000010}+f_{\foreseven{.1}{.5}0001000})+f_{\foreseven{.1}{.5}0000100}$
&$\mr C_3$&folding of $\mr A_5$\\[\dynshift]
\hline
$\mr A_5''$&10 &$(f_{\foreseven{.1}{.5}0000001}+f_{\foreseven{.1}{.5}0100000}) +(f_{\foreseven{.1}{.5}0000010}+f_{\foreseven{.1}{.5}0001000})+f_{\foreseven{.1}{.5}0000100}$
&$\mr C_3$&folding of $\mr A_5$\\
$\mr A_5+\mr A_1$\hfill($=(\mr A_5+\mr A_1)''$)&&\hfill${}+f_{\foreseven{.1}{.5}1000000}$&\hfill${}+\mr C_1$&\hfill+regular\\[\dynshift]
\hline
$\mr D_6(a_2)$&10&$f_{\foreseven{.1}{.5}0000001}+f_{\foreseven{.1}{.5}0000011}+f_{\foreseven{.1}{.5}0000110}+f_{\foreseven{.1}{.5}0001100} +f_{\foreseven{.1}{.5}0011000}+f_{\foreseven{.1}{.5}0100000}$&$\mr D_6(a_2)$&regular\\[\dynshift]
\hline
$\mr E_6(a_3)$\hfill($=(\mr A_5+\mr A_1)'$)&10&$(f_{\foreseven{.1}{.5}0000010}+f_{\foreseven{.1}{.5}1000000})+(f_{\foreseven{.1}{.5}0000100}+f_{\foreseven{.1}{.5}0010000}) +f_{\foreseven{.1}{.5}0001000}+f_{\foreseven{.1}{.5}1111110}$&$\mr F_4(a_2)$&folding of $\mr E_6$\\[\dynshift]
\hline
$\mr E_7(a_5)$\hfill($=\mr D_6(a_2)+\mr A_1$)&10&$f_{\foreseven{.1}{.5}0000001}+f_{\foreseven{.1}{.5}0000010}+f_{\foreseven{.1}{.5}0000100}+f_{\foreseven{.1}{.5}0001000} +f_{\foreseven{.1}{.5}0100000}+f_{\foreseven{.1}{.5}1000000}+f_{\foreseven{.1}{.5}1234321}$&$\mr E_7(a_5)$&regular\\[\dynshift]
\hline
$\mr A_6$&12&$f_{\foreseven{.1}{.5}0000001}+f_{\foreseven{.1}{.5}0000010}+f_{\foreseven{.1}{.5}0000100}+f_{\foreseven{.1}{.5}0001000} +f_{\foreseven{.1}{.5}0010000}+f_{\foreseven{.1}{.5}1000000}$&$\mr A_6$&regular\\[\dynshift]
\hline
$\mr D_5$&14&$f_{\foreseven{.1}{.5}1000000}+f_{\foreseven{.1}{.5}0010000}+f_{\foreseven{.1}{.5}0001000}+(f_{\foreseven{.1}{.5}0000100} +f_{\foreseven{.1}{.5}0100000})$&$\mr B_4$&folding of $\mr D_5$\\
$\mr D_5+\mr A_1$&&\hfill${}+f_{\foreseven{.1}{.5}0000001}$&\hfill${}+\mr C_1$&\hfill+regular\\
$\mr D_6(a_1)$&&\hfill${}+(f_{\foreseven{.1}{.5}1223221}-f_{\foreseven{.1}{.5}1123321})$&\hfill${}+\mr C_1'$&\hfill+IV$_{\mr B_4\subset\mr E_7}$\\
$\mr E_7(a_4)$\hfill($=\mr D_6(a_1)+\mr A_1$)&&\hfill${}+(f_{\foreseven{.1}{.5}1223221}-f_{\foreseven{.1}{.5}1123321})+f_{\foreseven{.1}{.5}0000001}$&\hfill${}+\mr C_1'+\mr C_1$&\hfill+IV$_{\mr B_4\subset\mr E_7}${}+regular\\[\dynshift]
\hline
$\mr E_6(a_1)$&16&$f_{\foreseven{.1}{.5}0000010}+f_{\foreseven{.1}{.5}0000100}+f_{\foreseven{.1}{.5}0010000}+f_{\foreseven{.1}{.5}1000000} +f_{\foreseven{.1}{.5}0101000}+f_{\foreseven{.1}{.5}0001100}$&$\mr E_6(a_1)$&regular\\[\dynshift]
\hline
$\mr D_6$&18&$f_{\foreseven{.1}{.5}0000001}+f_{\foreseven{.1}{.5}0000010}+f_{\foreseven{.1}{.5}0000100}+f_{\foreseven{.1}{.5}0001000} +(f_{\foreseven{.1}{.5}0010000}+f_{\foreseven{.1}{.5}0100000})$&$\mr B_5$&folding of $\mr D_6$\\
$\mr E_7(a_3)$\hfill$(=\mr D_6+\mr A_1)$&&\hfill${}+f_{\foreseven{.1}{.5}2234321}$&\hfill${}+\mr C_1$&\hfill+regular\\[\dynshift]
\hline
$\mr E_6$&22&$(f_{\foreseven{.1}{.5}0000010}+f_{\foreseven{.1}{.5}1000000})+(f_{\foreseven{.1}{.5}0000100}+f_{\foreseven{.1}{.5}0010000}) +f_{\foreseven{.1}{.5}0001000} +f_{\foreseven{.1}{.5}0100000}$&$\mr F_4$&folding of $\mr E_6$\\
$\mr E_7(a_2)$&&\hfill${}+(f_{\foreseven{.1}{.5}0112221}+f_{\foreseven{.1}{.5}1112211}+f_{\foreseven{.1}{.5}1122111})$&\hfill${}+\mr C_1$&\hfill+IV$_{\mr F_4\subset\mr E_7}$\\[\dynshift]
\hline
$\mr E_7(a_1)$&26&$f_{\foreseven{.1}{.5}0000001}+f_{\foreseven{.1}{.5}0000010}+f_{\foreseven{.1}{.5}0000100}+f_{\foreseven{.1}{.5}0010000} +f_{\foreseven{.1}{.5}1000000}+f_{\foreseven{.1}{.5}0011000}+f_{\foreseven{.1}{.5}0101000}$&$\mr E_7(a_1)$&regular\\[\dynshift]
\hline
$\mr E_7$&34&$f\pr$&$\mr E_7$&regular
\end{tabular}
}
\end{table}

\begin{table}[H]\caption{Normal forms of nilpotent elements in $\mr E_8$\label{tab:E8}}

\ \def\arraystretch{1.5}

\setlength\dynshift{.25em}

\resizebox{.95\textwidth}{!}{
\begin{tabular}{l|c|l|l|l}
$f$&$d$&representative&normal form&embedding\\
\hline
\hline
$\mr A_1$&2&$f_{\foreight{.5}00000001}$&$\mr C_1$&regular\\[\dynshift]
\hline
$2\mr A_1$&2&$f_{\foreight{.5}00000001}+f_{\foreight{.5}00000100}$&$2\mr C_1$&regular\\[\dynshift]
\hline
$3\mr A_1$&3&$f_{\foreight{.5}00000001}+f_{\foreight{.5}00000100}+f_{\foreight{.5}00010000}$&$3\mr C_1$&regular\\[\dynshift]
\hline
$4\mr A_1$&3&$f_{\foreight{.5}00000001}+f_{\foreight{.5}00000100}+f_{\foreight{.5}00010000}+f_{\foreight{.5}10000000}$&$4\mr C_1$&regular\\[\dynshift]
\hline
$\mr A_2$&4&$f_{\foreight{.5}00000001}+f_{\foreight{.5}00000010}$&$\mr A_2$&regular\\
$\mr A_2+\mr A_1$&&\hfill${}+f_{\foreight{.5}10000000}$&\hfill${}+\mr C_1$&regular\\
$\mr A_2+2\mr A_1$&&\hfill${}+f_{\foreight{.5}01000000}+f_{\foreight{.5}10000000}$&\hfill${}+2\mr C_1$&regular\\
$\mr A_2+3\mr A_1$&&\hfill${}+f_{\foreight{.5}00001000}+f_{\foreight{.5}01000000}+f_{\foreight{.5}10000000}$&\hfill${}+3\mr C_1$&regular\\[\dynshift]
\hline
$2\mr A_2$&4&$[f_{\foreight{.5}00000100}+f_{\foreight{.5}00001000}] +[f_{\foreight{.5}00100000}+f_{\foreight{.5}10000000}]$&$2\mr A_2$&regular\\[\dynshift]
\hline
$2\mr A_2+\mr A_1$&5&$[f_{\foreight{.5}00000100}+f_{\foreight{.5}00001000}] +[f_{\foreight{.5}00100000}+f_{\foreight{.5}10000000}]$\hfill${}+f_{\foreight{.5}00000001}$&$2\mr A_2$\hfill${}+\mr C_1$&regular\\[\dynshift]
\hline
$2\mr A_2+2\mr A_1$&5&$[f_{\foreight{.5}00000100}+f_{\foreight{.5}00001000}] +[f_{\foreight{.5}00100000}+f_{\foreight{.5}10000000}]$\hfill${}+f_{\foreight{.5}01000000}+f_{\foreight{.5}00000001}$&$2\mr A_2$\hfill${}+2\mr C_1$&regular\\[\dynshift]
\hline
$\mr A_3$&6&$(f_{\foreight{.5}00000001}+f_{\foreight{.5}00000100})+f_{\foreight{.5}00000010}$&$\mr C_2$&folding of $\mr A_3$\\
$\mr A_3+\mr A_1$&&\hfill${}+f_{\foreight{.5}10000000}$&\hfill${}+\mr A_1$&\hfill+regular\\
$\mr A_3+2\mr A_1$&&\hfill${}+f_{\foreight{.5}00010000}+f_{\foreight{.5}10000000}$&\hfill${}+2\mr A_1$&\hfill+regular\\[\dynshift]
\hline
$\mr D_4(a_1)$&6&$f_{\foreight{.5}00001000}+f_{\foreight{.5}00011000}+f_{\foreight{.5}00110000}+f_{\foreight{.5}01000000}$&$\mr D_4(a_1)$&regular\\
$\mr D_4(a_1)+\mr A_1$&&$\hfill+f_{\foreight{.5}00000001}$&\hfill${}+\mr C_1$&regular\\
$\mr A_3+\mr A_2$\hfill$(=\mr D_4(a_1)+2\mr A_1)$&&\hfill${}+f_{\foreight{.5}01122221}+f_{\foreight{.5}00000001}$&\hfill${}+2\mr C_1$&regular\\
$\mr A_3+\mr A_2+\mr A_1$\hfill$(=\mr D_4(a_1)+3\mr A_1)$&&\hfill${}+f_{\foreight{.5}22343221}+f_{\foreight{.5}01122221}+f_{\foreight{.5}00000001}$&\hfill${}+3\mr C_1$&regular\\
$\mr D_4(a_1)+\mr A_2$&&\hfill${}+f_{\foreight{.5}00000010}+f_{\foreight{.5}00000001}$&\hfill${}+\mr A_2$&regular\\[\dynshift]
\hline
$2\mr A_3$&7&$[(f_{\foreight{.5}00000001}+f_{\foreight{.5}00000100})+f_{\foreight{.5}00000010}] +[(f_{\foreight{.5}00010000}+f_{\foreight{.5}10000000})+f_{\foreight{.5}00100000}]$&$2\mr C_2$&2 foldings of $\mr A_3$\\[\dynshift]
\hline
$\mr A_4$&8&$f_{\foreight{.5}00000001}+f_{\foreight{.5}00000010}+f_{\foreight{.5}00000100}+f_{\foreight{.5}00001000}$&$\mr A_4$&regular\\
$\mr A_4+\mr A_1$&&\hfill${}+f_{\foreight{.5}10000000}$&\hfill${}+\mr C_1$&regular\\
$\mr A_4+2\mr A_1$&&\hfill${}+f_{\foreight{.5}01000000}+f_{\foreight{.5}10000000}$&\hfill${}+2\mr C_1$&regular\\
$\mr A_4+\mr A_2$&&\hfill${}+f_{\foreight{.5}00100000}+f_{\foreight{.5}10000000}$&\hfill${}+\mr A_2$&regular\\
$\mr A_4+\mr A_2+\mr A_1$&&\hfill${}+[f_{\foreight{.5}00100000}+f_{\foreight{.5}10000000}]+f_{\foreight{.5}01000000}$&\hfill${}+\mr A_2+\mr C_1$&regular\\[\dynshift]
\hline
$\mr A_4+\mr A_3$&9&$[f_{\foreight{.5}00000001}+f_{\foreight{.5}00000010}+f_{\foreight{.5}00000100}+f_{\foreight{.5}00001000}]
+[(f_{\foreight{.5}12354321}+f_{\foreight{.5}00100000})+f_{\foreight{.5}10000000}]$&$\mr A_4$\hfill${}+\mr C_2$&regular\hfill+folding of $\mr A_3$\\[\dynshift]
\hline
$\mr D_4$&10&$(f_{\foreight{.5}00001000}+f_{\foreight{.5}01000000}+f_{\foreight{.5}00100000})+f_{\foreight{.5}00010000}$&$\mr G_2$&folding of $\mr D_4$\\
$\mr D_4+\mr A_1$&&\hfill${}+f_{\foreight{.5}00000001}$&\hfill${}+\mr C_1$&\hfill+regular\\
$\mr D_5(a_1)$\hfill($=\mr D_4+2\mr A_1$)&&\hfill${}+f_{\foreight{.5}00000011}+f_{\foreight{.5}00000001}$&\hfill${}+2\mr C_1$&\hfill+regular\\
$\mr D_5(a_1)+\mr A_1$\hfill($=\mr D_4+3\mr A_1$)&&\hfill${}+f_{\foreight{.5}01122221}+f_{\foreight{.5}00000011}+f_{\foreight{.5}00000001}$&\hfill${}+3\mr C_1$&\hfill+regular\\
$\mr D_4+\mr A_2$&&\hfill${}+f_{\foreight{.5}00000010}+f_{\foreight{.5}00000001}$&\hfill${}+\mr A_2$&\hfill+regular\\
$\mr D_5(a_1)+\mr A_2$&&\hfill${}+[f_{\foreight{.5}00000010}+f_{\foreight{.5}00000001}] +(f_{\foreight{.5}11110000}+f_{\foreight{.5}10111000})$&\hfill${}+\mr A_2+\mr C_1$&\hfill+regular+IV$_{\mr F_4\subset\mr E_8}$\\[\dynshift]
\hline
$\mr A_5$&10&$(f_{\foreight{.5}00000001}+f_{\foreight{.5}00010000})+(f_{\foreight{.5}00000010}+f_{\foreight{.5}00001000})+f_{\foreight{.5}00000100}$&$\mr C_3$&folding of $\mr A_5$\\
$\mr A_5+\mr A_1$\hfill($=(\mr A_5+\mr A_1)'$)&&\hfill${}+f_{\foreight{.5}10000000}$&\hfill${}+\mr C_1$&\hfill+regular\\[\dynshift]
\hline
$\mr E_6(a_3)$\hfill($=(\mr A_5+\mr A_1)''$)&10&$(f_{\foreight{.5}00000100}+f_{\foreight{.5}10000000})+(f_{\foreight{.5}00001000}+f_{\foreight{.5}00100000}) +f_{\foreight{.5}00010000}+f_{\foreight{.5}11111100}$&$\mr F_4(a_2)$&folding of $\mr E_6$\\
$\mr E_6(a_3)+\mr A_1$\hfill($=\mr A_5+2\mr A_1$)&&\hfill${}+f_{\foreight{.5}00000001}$&\hfill${}+\mr C_1$&\hfill+regular\\[\dynshift]
\hline
$\mr D_6(a_2)$&10&$f_{\foreight{.5}00000010}+f_{\foreight{.5}00000110}+f_{\foreight{.5}00001100}+f_{\foreight{.5}00011000}
+f_{\foreight{.5}00110000}+f_{\foreight{.5}01000000}$&$\mr D_6(a_2)$&regular\\[\dynshift]
\hline
$\mr E_7(a_5)$\hfill($=\mr A_5+\mr A_2$)&10&$f_{\foreight{.5}00000010}+f_{\foreight{.5}00000100}+f_{\foreight{.5}00001000}+f_{\foreight{.5}00010000} +f_{\foreight{.5}01000000}+f_{\foreight{.5}10000000}+f_{\foreight{.5}12343210}$&$\mr E_7(a_5)$&regular\\[\dynshift]
\hline
$\mr E_8(a_7)$\hfill($=2\mr A_4$)&10&$f_{\foreight{.5}12354321}+f_{\foreight{.5}00000001}+f_{\foreight{.5}00000010}+f_{\foreight{.5}00000100}
+f_{\foreight{.5}00001000}+f_{\foreight{.5}00100000}+f_{\foreight{.5}10000000}+f_{\foreight{.5}01000000}$&$\mr E_8(a_7)$&regular\\[\dynshift]
\hline
$\mr A_6$&12&$f_{\foreight{.5}00000001}+f_{\foreight{.5}00000010}+f_{\foreight{.5}00000100}+f_{\foreight{.5}00001000}+
f_{\foreight{.5}00010000}+f_{\foreight{.5}01000000}$&$\mr A_6$&regular\\
$\mr A_6+\mr A_1$&&\hfill${}+f_{\foreight{.5}10000000}$&\hfill${}+\mr C_1$&regular\\[\dynshift]
\hline
$\mr D_5$&14&$f_{\foreight{.5}10000000}+f_{\foreight{.5}00100000}+f_{\foreight{.5}00010000}+(f_{\foreight{.5}00001000}+f_{\foreight{.5}01000000})$&$\mr B_4$&folding of $\mr D_5$\\
$\mr D_5+\mr A_1$&&\hfill${}+f_{\foreight{.5}00000010}$&\hfill${}+\mr C_1$&\hfill+regular\\
$\mr D_6(a_1)$\hfill($=\mr D_5+2\mr A_1$)&&\hfill${}+f_{\foreight{.5}23465432}+f_{\foreight{.5}00000010}$&\hfill${}+2\mr C_1$&\hfill+regular\\
$\mr D_5+\mr A_2$&&\hfill${}+f_{\foreight{.5}00000010}+f_{\foreight{.5}00000001}$&\hfill${}+\mr A_2$&\hfill+regular\\
$\mr D_7(a_2)$\hfill($=\mr D_5+\mr A_3$)&&\hfill${}+f_{\foreight{.5}00000001}+(f_{\foreight{.5}12233210}-f_{\foreight{.5}11233210})$&\hfill${}+\mr B_2$
&\hfill+IV$_{\mr B_4\subset\mr E_8}$\\
$\mr E_7(a_4)$\hfill($=\mr D_6(a_1)+\mr A_1$)&&\hfill${}+f_{\foreight{.5}00000010}+(f_{\foreight{.5}12233210}-f_{\foreight{.5}11233210})$&\hfill${}+\mr C_1+\tilde{\mr C}_1$&\hfill+IV$_{\mr B_4\subset\mr E_8}$
\end{tabular}
}
\end{table}

\begin{center}
Table 7. Normal forms of nilpotent elements in $\mr E_8$ (cont'd.)

\ \def\arraystretch{1.5}

\setlength\dynshift{.25em}

\resizebox{.95\textwidth}{!}{
\begin{tabular}{l|c|l|l|l}
$f$&$d$&representative&normal form&embedding\\
\hline
\hline
$\mr A_7$&15&$(f_{\foreight{.5}00000001}+f_{\foreight{.5}10000000})+(f_{\foreight{.5}00000010}+f_{\foreight{.5}00100000})
+(f_{\foreight{.5}00000100}+f_{\foreight{.5}00010000})+f_{\foreight{.5}00001000}$&$\mr C_4$&folding of $\mr A_7$\\[\dynshift]
\hline
$\mr E_6(a_1)$&16&$f_{\foreight{.5}00000100}+f_{\foreight{.5}00001000}+f_{\foreight{.5}00100000}+f_{\foreight{.5}10000000} +f_{\foreight{.5}01010000}+f_{\foreight{.5}00011000}$&$\mr E_6(a_1)$&regular\\
$\mr E_6(a_1)+\mr A_1$&&\hfill${}+f_{\foreight{.5}00000001}$&\hfill${}+\mr C_1$&regular\\
$\mr E_8(b_6)$\hfill($=\mr E_6(a_1)+\mr A_2=\mr D_8(a_3)$)&&\hfill${}+f_{\foreight{.5}23465431}+f_{\foreight{.5}00000001}$&\hfill${}+\mr A_2$&regular\\[\dynshift]
\hline
$\mr D_6$&18&$f_{\foreight{.5}00000010}+f_{\foreight{.5}00000100}+f_{\foreight{.5}00001000}+f_{\foreight{.5}00010000}
+(f_{\foreight{.5}00100000}+f_{\foreight{.5}01000000})$&$\mr B_5$&folding of $\mr D_6$\\
$\mr E_7(a_3)$\hfill($=\mr D_6+\mr A_1$)&&\hfill${}+f_{\foreight{.5}23465432}$&\hfill${}+\mr C_1$&\hfill+regular\\
$\mr D_7(a_1)$\hfill($=\mr D_6+2\mr A_1$)&&\hfill${}+f_{\foreight{.5}22343210}+f_{\foreight{.5}23465432}$&\hfill${}+2\mr C_1$&\hfill+regular\\[\dynshift]
\hline
$\mr E_8(a_6)$\hfill($=\mr A_8$)&18&$f_{\foreight{.5}00000001}+f_{\foreight{.5}00000010}+f_{\foreight{.5}00000100}+f_{\foreight{.5}00001000}
+f_{\foreight{.5}00010000}+f_{\foreight{.5}00100000}+f_{\foreight{.5}10000000}+f_{\foreight{.5}12243210}$&$\mr E_8(a_6)$&regular\\[\dynshift]
\hline
$\mr E_6$&22&$(f_{\foreight{.5}00000100}+f_{\foreight{.5}10000000})
+(f_{\foreight{.5}00001000}+f_{\foreight{.5}00100000})+f_{\foreight{.5}00010000}+f_{\foreight{.5}01000000}$&$\mr F_4$&folding of $\mr E_6$\\
$\mr E_6+\mr A_1$&&\hfill$+f_{\foreight{.5}00000001}$&\hfill${}+\mr C_1$&\hfill+regular\\
$\mr E_7(a_2)$&&\hfill$+(f_{\foreight{.5}01122210}+f_{\foreight{.5}11122110}+f_{\foreight{.5}11221110})$&\hfill${}+\tilde{\mr C}_1$&\hfill+IV$_{\mr F_4\subset\mr E_8}$\\
$\mr E_8(b_5)$\hfill($=\mr E_7(a_2)+\mr A_1$)&&\hfill$+f_{\foreight{.5}23465431}+f_{\foreight{.5}00000001}$&\hfill${}+\mr A_2$&\hfill+IV$_{\mr F_4\subset\mr E_8}$\\[\dynshift]
\hline
$\mr D_7$&22&$f_{\foreight{.5}00000001}+f_{\foreight{.5}00000010}+f_{\foreight{.5}00000100}+f_{\foreight{.5}00001000}
+f_{\foreight{.5}00010000}+(f_{\foreight{.5}00100000}+f_{\foreight{.5}01000000})$&$\mr B_6$&folding of $\mr D_7$\\[\dynshift]
\hline
$\mr E_8(a_5)$\hfill($=\mr D_8(a_1)$)&22&$f_{\foreight{.5}22343210}+f_{\foreight{.5}00000001}+f_{\foreight{.5}00000010}+f_{\foreight{.5}00000100}+f_{\foreight{.5}00001000}
+f_{\foreight{.5}00011000}+f_{\foreight{.5}00110000}+f_{\foreight{.5}01000000}$&$\mr E_8(a_5)$&regular\\[\dynshift]
\hline
$\mr E_7(a_1)$&26&$f_{\foreight{.5}00000010}+f_{\foreight{.5}00000100}+f_{\foreight{.5}00001000}+f_{\foreight{.5}00100000} +f_{\foreight{.5}10000000}+f_{\foreight{.5}00110000}+f_{\foreight{.5}01010000}$&$\mr E_7(a_1)$&regular\\
$\mr E_8(b_4)$\hfill($=\mr E_7(a_1)+\mr A_1$)&&\hfill${}+f_{\foreight{.5}23465432}$&\hfill${}+\mr C_1$&regular\\[\dynshift]
\hline
$\mr E_8(a_4)$\hfill($=\mr D_8$)&28&$f_{\foreight{.5}22343210}+f_{\foreight{.5}00000001}+f_{\foreight{.5}00000010}+f_{\foreight{.5}00000100}+f_{\foreight{.5}00001000}
+f_{\foreight{.5}00010000}+f_{\foreight{.5}00100000}+f_{\foreight{.5}01000000}$&$\mr E_8(a_4)$&regular\\[\dynshift]
\hline
$\mr E_7$&34&$f_{\foreight{.5}00000010}+f_{\foreight{.5}00000100}+f_{\foreight{.5}00001000}
+f_{\foreight{.5}00010000}+f_{\foreight{.5}00100000}+f_{\foreight{.5}01000000}+f_{\foreight{.5}10000000}$&$\mr E_7$&regular\\
$\mr E_8(a_3)$\hfill($=\mr E_7+\mr A_1$)&&\hfill${}+f_{\foreight{.5}23465432}$&\hfill${}+\mr C_1$&regular\\[\dynshift]
\hline
$\mr E_8(a_2)$&38&$f_{\foreight{.5}00000001}+f_{\foreight{.5}00000010}+f_{\foreight{.5}00001100}+f_{\foreight{.5}00011000} +f_{\foreight{.5}00110000}+f_{\foreight{.5}01010000}+f_{\foreight{.5}00100000}+f_{\foreight{.5}10000000}$&$\mr E_8(a_2)$&regular\\[\dynshift]
\hline
$\mr E_8(a_1)$&46&$f_{\foreight{.5}00000001}+f_{\foreight{.5}00000010}+f_{\foreight{.5}00000100}+f_{\foreight{.5}00001000}
+f_{\foreight{.5}00100000}+f_{\foreight{.5}00110000}+f_{\foreight{.5}01010000}+f_{\foreight{.5}10000000}$&$\mr E_8(a_1)$&regular\\[\dynshift]
\hline
$\mr E_8$&58&$f\pr$&$\mr E_8$&regular
\end{tabular}
}
\end{center}

\section{A map from nilpotent orbits to conjugacy classes in the Weyl group}

Recall the following construction (\cite{EKV}, cf. \cite{K} and \cite{S}).
Let $\mf g$ be a simple Lie algebra and let $W$ be its Weyl group.
Let $f$ be a nilpotent element of $\mf g$ of regular semisimple type.
This means that there exists $E\in\mf g_d$ (in the $\bb Z$-grading \eqref{eq:grading}),
such that $f+E$ is a regular semisimple element of $\mf g$.
Its centralizer is a Cartan subalgebra $\mf h'$ of $\mf g$.

Let $s_i=\alpha_i(h)$, $i=1,...,r$, be the Dynkin labels for $f$, and let $s_0=2$. Let
\[
m=\sum_{i=0}^ra_is_i,
\]
where $\sum_{i=1}^ra_i\alpha_i$ is the highest root, and let $\eps=e^{\frac{2\pi i}m}$.
Define an inner automorphism $\sigma_f$ of $\mf g$ by letting
\begin{equation}\label{eq:sigma}
\sigma_f(e_{\alpha_i})=\eps^{s_i}e_{\alpha_i},\quad\sigma_f(e_{-\alpha_i})=\eps^{-s_i}e_{-\alpha_i},\qquad i=1,...,r.
\end{equation}
The order of $\sigma_f$ is $m$ if $f$ is not even and $m/2$ if $f$ is even.

It was pointed out in \cite{EKV} that all $f$ of regular semisimple type are even,
with the exception of the nilpotent elements in $\mf{sl}_n$, associated to partitions $(k^{(m)},1)$,
where $k$ is even (hence $n=mk+1$ is odd); but these $f$ are even in $\mf{sl}_{n-1}$, hence we may assume that $f$ is even.
Then all the $s_i$ are either $0$ or $2$, and we replace the $2$'s by the $1$'s.
Thus, $\sigma_f$ can be depicted using an extended Dynkin diagram of $f$ by dividing the labels by 2 and letting $s_0=1$.

Since $f+E$ is an eigenvector of $\sigma_f$ with eigenvalue $\eps^{-1}$,
the Cartan subalgebra $\mf h'$ is $\sigma_f$-invariant, hence $\sigma_f$ induces an element $w_f$ of the Weyl group $W$.

Since all irreducible nilpotent elements are of regular semisimple type and even, we thus obtain a map
from the set of orbits of irreducible nilpotent elements in $\mf g$ to the set of conjugacy classes in $W$.
This map is depicted in Table \ref{tab:w} below.

Taking the normal form of an arbitrary nilpotent element
\[
f=\sum_jf[j],\quad f[j]\in\mf g[j],
\]
we can extend the map, given by Table \ref{tab:w}, to all $f$ by letting
\[
w_f=\prod_jw_{f[j]},
\]
where $w_{f[j]}$ is the image in $W$ of the element of the Weyl group of $\mf g[j]$ under the embedding $\mf g[j]\subset\mf g$.

Note, however, that, though the map $f\mapsto w_f$ coincides with that of Kazhdan-Lusztig \cite{KL}, \cite{Spa}
on the set of regular semisimple type $f$, in general it is different. We will study this question in a subsequent publication.

\begin{table}[H]\caption{Diagrams of $w\in W$, corresponding to irreducible nilpotent elements in simple $\mf g$}\label{tab:w}

\

\begin{tabular}{c|c|c|c}
$f$&diagram of $w_f$&order of $w_f$&characteristic polynomial of $w_f$\\
\hline
$\mr A_{2k}$&\oneADynkin{1,1,1,{\,\cdots},1,1}&$2k+1$&$x^{2k}+x^{2k-1}+...+1$\\
$\mr C_k$&\oneCDynkin{1,1,1,{\,\cdots},1}1&$2k$&$x^k+1$\\
$\mr B_k$&\oneBDynkin{1,1,1,1,{\,\cdots},1}1&$2k$&$x^k+1$\\
$\mr D_{2k+2}(a_k)$&\oneDDynkin{1,1,0,1,0,{\,\cdots},1,0}11&$2k+2$&$(x^{k+1}+1)^2$\\
$\mr G_2$&\oneGDynkin{1,1,1}&6&$\phi_6$\\
$\mr F_4$&\oneFDynkin{1,1,1,1,1}&12&$\phi_{12}$\\
$\mr F_4(a_2)$&\oneFDynkin{1,0,1,0,1}&6&$\phi_6^2$\\
$\mr E_6(a_1)$&\onEsixDynkin{1,1,0,1,1,1,1}&9&$\phi_9$\\
$\mr E_7$&\onEsevenDynkin{1,1,1,1,1,1,1,1}&18&$\phi_{18}\phi_2$\\
$\mr E_7(a_1)$&\onEsevenDynkin{1,1,1,0,1,1,1,1}&14&$\phi_{14}\phi_2$\\
$\mr E_7(a_5)$&\onEsevenDynkin{1,0,0,1,0,0,1,0}&6&$\phi_6^3\phi_2$\\
$\mr E_8$&\onEightDynkin{1,1,1,1,1,1,1,1,1}&30&$\phi_{30}$\\
$\mr E_8(a_1)$&\onEightDynkin{1,1,1,1,1,0,1,1,1}&24&$\phi_{24}$\\
$\mr E_8(a_2)$&\onEightDynkin{1,1,1,0,1,0,1,1,1}&20&$\phi_{20}$\\
$\mr E_8(a_4)$&\onEightDynkin{1,1,0,1,0,1,0,1,0}&15&$\phi_{15}$\\
$\mr E_8(a_5)$&\onEightDynkin{1,0,1,0,0,1,0,1,0}&12&$\phi_{12}^2$\\
$\mr E_8(a_6)$&\onEightDynkin{1,0,1,0,0,1,0,0,0}&10&$\phi_{10}^2$\\
$\mr E_8(a_7)$&\onEightDynkin{1,0,0,0,1,0,0,0,0}&6&$\phi_6^4$
\end{tabular}
\end{table}

\end{document}